\documentclass[final, 11pt,reqno]{amsart}
\usepackage{amssymb,amsmath,graphicx,amsfonts,euscript}
\usepackage{color}
\usepackage[pagewise]{lineno}
\usepackage{showkeys}
\usepackage{fancyhdr}
\usepackage[margin=1in]{geometry}
\usepackage{titletoc}
\usepackage[linktoc=page]{hyperref}
\hypersetup{colorlinks=false, pdfstartview=FitV, linkcolor=blue,citecolor=blue, urlcolor=blue}
\newtheorem{theorem}{\bf Theorem}[section]

\newtheorem{proposition}[theorem]{\bf Proposition}

\newtheorem{lemma}[theorem]{\bf Lemma}

\newtheorem{assumption}[theorem]{Assumption}

\newcommand{\beq}{\begin{equation}}
\newcommand{\eeq}{\end{equation}}
\newcommand{\ben}{\begin{eqnarray}}
\newcommand{\een}{\end{eqnarray}}
\newcommand{\beno}{\begin{eqnarray*}}
\newcommand{\eeno}{\end{eqnarray*}}

\numberwithin{equation}{section}
\makeatother
\title[Enhanced dissipation and blow-up suppression]{Enhanced dissipation and blow-up suppression for the three dimensional Keller-Segel equation with a non-shear incompressible flow}

\author[B.Shi]{Binbin Shi}
\address{School of Mathematics and Statistics, Nanjing University of Science and Technology, Nanjing, 210094, P.R.  China.}
\email{shibb@njust.edu.cn}

\author[W.Wang ]{Weike Wang}
\address{School of Mathematical Sciences, CMA-Shanghai and Institute of Natural Science, Shanghai Jiao Tong University, Shanghai, 200240, P.R.China.}
\email{wkwang@sjtu.edu.cn}

\begin{document}
%\vskip .2in
\begin{abstract}
In this paper, we consider the Cauchy problem for the three dimensional parabolic-elliptic Keller-Segel equation with a large non-shear incompressible flow. Without advection, there exist solution with arbitrarily mass which blow up in finite time. Firstly, we introduce a three dimensional non-shear incompressible flow and study the enhanced dissipation of such flows by resolvent estimate method. Next, we show that the enhanced dissipation of  such flow can suppress blow-up of solution to three dimensional parabolic-elliptic Keller-Segel equation and establish global classical solution with large initial data.
%In this paper, we introduce a three dimensional non-shear incompressible flow and study the dissipation enhancement %of such flows by resolvent estimate method. And we consider the three dimensional ellipitic-parabolic Keller-Segel %equation with a large non-shear incompressible flow. Without the such flows, there exist solution with arbitrarily %mass which blow up in finite time. We show that such flow can suppress blow-up of solution to three dimensional %ellipitic-parabolic Keller-Segel equation and establish global classical solution with large initial data.
\end{abstract}
\maketitle
%\vskip .1in
\section{Introduction and main results}\label{sec.1}
%\vskip .1in
In this paper, we consider the following three dimensional parabolic-elliptic Keller-Segel equation with an incompressible flow on domain $\mathbb{T}^2\times \mathbb{R}=\{x=(x_1,x_2,y)\big|x_1,x_2\in \mathbb{T}_{2\pi}, y\in \mathbb{R}\}$
\begin{equation}\label{eq:1.1}
\begin{cases}
\partial_tn+A{\bf u}\cdot \nabla n-\Delta n+\nabla\cdot(n \nabla c)=0,\\
-\Delta c=n-c,\\
n(t,x)\big|_{t=0}=n_0(x), \ \ \   (t,x)\in \mathbb{R}^+\times \mathbb{T}^2\times \mathbb{R}.
\end{cases}
\end{equation}
Here $n(t,x)$ and $c(t,x)$ denote the micro-oranism density and the chemo-attractant density respectively. The divergence free vector field ${\bf u}$ represents the underlying fluid velocity and $A$ is a positive constant. A non-shear incompressible flow ${\bf u}$ is considered in this paper, which is defined as follows
%We consider a non-shear incompressible flow, which is defined as follows
\begin{equation}\label{eq:1.2}
{\bf u}=(y,y^2,0),
\end{equation}
and we study the global well-posedness of Equation \eqref{eq:1.1} by large advection.

\vskip .05in

If $A=0$, the Equation \eqref{eq:1.1} is  reduced to the classical Keller-Segel equation, and the study of equation has gained widespread attention in the past few decades.
%and the theory of  Keller-Segel equation gained some attention in the past few decades.
It is well know that the Keller-Segel equation in dimensions larger than one, the solution may blow up in finite time. More precisely, for two dimensional case, if $L^1$ norm of initial data $n_0$ is less than $8\pi$, there exists a unique global solution, if $L^1$ norm of initial data $n_0$ exceeds $8\pi$, the solution may blow up in finite time, the related research can refer to ~\cite{Herrero.1996,Herrero.1997,Nagai.1995}. In the higher dimensional case, the blow-up may occurs for solutions with arbitrary small $L^1$ norm of the initial data, see \cite{Corrias.2004,Winkler.2019}. For the parabolic-parabolic Keller-Segel equation, some similar results can read \cite{Calvez.2008,Schweyer.2014,Winkler.2013}.

\vskip .05in

In fact, it is a more realistic scenario that chemotactic processes take place in a moving fluid. The possible effects result from the interaction of chemotactic and fluid transport process, the related problem have been studied by many authors, see \cite{Che.2016, Liu.2011, Lorz.2010, Wang.2019, Winkler.201201}. %\cite{ Che.2016,Francesco.2010, Duan.2010, Liu.2011, Lorz.2010, Wang.2019, Winkler.201201}.
The Keller-Segel equation with incompressible flow is one of many attempts to take into account the effect of the moving fluid. An interesting question arises whether one can suppress the finite time blow-up by the stabilizing effect of moving fluid. In the recent years, some progresses have been made in proving the suppression of the chemotactic  blow-up by presence of fluid flow. %For the parabolic-elliptic case,
Kiselev and Xu \cite{Kiselev.2016} considered the relaxation enhancing flow which was introduced in \cite{Constantin.2008}, they proved that the solution of the advective Keller-Segel equation does not blow-up in finite time provided the amplitude of the relaxation enhancing flow is large enough. Later for the generalized  Keller-Segel equation with fractional Laplacian and relaxation enhancing flow, the global well-posedness were discussed in~\cite{Hopf.2018,Shi.2019}. Bedrossian and He~\cite{Bedrossian.2017} proved that the shear flows can also suppress the blow-up of solution to parabolic-elliptic Keller-Segel equation. More precisely, they proved that the solution is global in the two dimension case, while for the three dimension case, the global well-posedness is guaranteed only when the initial mass is less than $8\pi$. Feng, Shi and Wang \cite{Feng.2022} considered the three dimensional parabolic-elliptic Keller-Segel equation with large planar helical flow and obtained global solution with large data.
%For the three dimensional case, Feng, Shi and Wang \cite{Feng.2022} obtained global solution by a large planar helical %flow.
 %, which is also a non-shear flow.
%For the parabolic-parabolic case,
He \cite{He.2018} proved that the shear flows can also suppress the blow-up of the solution to parabolic-parabolic Keller-Segel equation in the case of the two dimension case and some smallness assumption. For chemotactic coupled with fluid equations, Zeng, Zhang and Zi \cite{Zeng.2021} considered the two dimensional Keller-Segel-Navier-Stokes equation near the Couette flow, they proved that the solution is global existence if $A$ is large enough.
%if $A$ is large enough, the solution is global in time.
%the solutions of the equation are global in time.%, which considered the parabolic-elliptic case and parabolic-parabolic case.

\vskip .05in

For additional flows of the ambient environment are taken into consideration to prevent the blow-up, the authors proved the enhanced dissipation effect of fluid. It means that the solution undergoes a large growth in its gradient, then dissipative term was large and can dominates the nonlinear effect.
%The analysis of enhanced dissipation effect to the incompressible flows, we can use the RAGE theorem %\cite{Constantin.2008}, hypocercivity \cite{Bedrossian.201701,Villani.2009} and resolvent estimate \cite{Wei.2021}. It %is worth mentioning that the enhanced dissipation effect of shear flow occurs only in one direction and can  in some %sense  suppress one dimension in parabolic-elliptic Keller-Segel equation, hence make two dimension global %well-posedness and three dimension global well-posedness in the case of the initial mass less than $8\pi$.
%\ \ \
%In addition to, the situation is  somewhat different in the case of the parabolic-parabolic Keller-Segel system, the shear flow has both enhanced dissipation effect and destabilizing effect in the system \eqref{eq:1.1}. On the one hand, same as parabolic-elliptic case, the shear flow has enhanced dissipation effect, which is a good effect. On the other hand, the advection $Au\cdot \nabla c$ in the system \eqref{eq:1.1} creates large gradient for the chemo-attractant density $c$.
Recently, the enhanced dissipation of shear flow has been widely studied, such as Couette flow \cite{Bedrossian.201602}, Poiseuille flow \cite{Coti.2020} and Kolmogorov flow \cite{Wei.201901,Wei.2020}. And some new analysis method was introduced, such as hypocercivity \cite{Bedrossian.201701} and resolvent estimate \cite{Wei.2021}.  It is worth mentioning that the enhanced dissipation effect of shear flow occurs only in the nonzero mode of $x_1$ and the
%in one direction and can in some sense  suppress one dimension in parabolic-elliptic Keller-Segel equation, hence make %two dimension global well-posedness and three dimension global well-posedness in the case of the initial mass less %than $8\pi$.
one dimensional parabolic-elliptic Keller-Segel equation is global, hence make two dimension global well-posedness and three dimension global well-posedness in the case of the initial mass less than $8\pi$.

\vskip .05in

%In this paper, we consider the system \eqref{eq:1.1}, the goal is to show that the blow-up solution of the three %dimensional parabolic-parabolic Keller-Segel systems can be suppressed through the planar helical flow. This question %is motivated by the works of Feng et.al. \cite{Feng.2022}, He \cite{He.2018}. We know that the shear flow can not %suppress blow-up of the three dimensional parabolic-elliptic Kelle-Segel system, because no enhanced dissipation part %of solution is a two dimensional system. In \cite{Feng.2022}, authors introduced the planar helical flow, and consider %the three dimensional parabolic-elliptic Keller-Segel system with planar helical flow. Then prove that no enhanced %dissipation part of solution is a two dimensional system in the three dimensional parabolic-elliptic Keller-Segel %system, and obtained the global existence. For parabolic-parabolic case, He \cite{He.2018} considered the two %dimensional Keller-Segel systems with shear flow, but for three dimensional case, the similar with parabolic-elliptic %Keller-Segel system, the shear flow can not suppress the blow-up solution of the parabolic-parabolic Keller-Segel %system. Hence we consider the three dimensional parabolic-parabolic Keller-Segel systems with planar helical flow in %this paper. Since the destabilizing effect of mixing in system \eqref{eq:1.1}, some smallness assumption on the %initial data is needed.

In this paper, we consider the Equation \eqref{eq:1.1}, the goal is to show that the blow-up solution of the three dimensional Keller-Segel equation can be suppressed through a non-shear incompressible flow in \eqref{eq:1.2}. We will prove the global existence of solution to three dimensional Keller-Segel equation with an advective flow. This question is motivated by works of Bedrossian et.~al \cite{Bedrossian.2017} and Feng et.~al \cite{Feng.2022}. We will prove that the $L^2$ norm of the solution is bounded uniformly in time by choosing the flow with large amplitude. Actually, since the $L^2$ estimate is supercritical estimate for the three dimensional Keller-Segel equation, see~\cite{Kiselev.2016}, the global classical solution can be achieved. To point out, Bedrossian and He~\cite{Bedrossian.2017} studied the global existence of Keller-Segel equation with a shear flow in dimension $d=3$, which is achieved when the initial mass is less than $8\pi$. In this paper, we prove the global existence with any initial data, instead with the $8\pi$ restriction.

\vskip .05in

For the enhanced dissipation of non-shear flow in \eqref{eq:1.2}, we consider the advection-diffusion equation
\begin{equation}\label{eq:001}
\partial_tf+A{\bf u}\cdot \nabla f-\Delta f=0,\ \ \ f(0,x)=f_0(x),
\end{equation}
where $x=(x_1,x_2,y)\in \mathbb{T}^2\times\mathbb{R}$. The enhanced dissipation effect of fluid is an interesting phenomenon, it means that the dissipation effect will be enhanced and the $L^2$ norm of solution to Equation \eqref{eq:001} has a faster decaying rate if $A$ is large enough. If we consider the incompressible flow ${\bf u}$ in \eqref{eq:1.2} and take the time scale, the Equation \eqref{eq:001} can be written as
\begin{equation}\label{eq:1.3}
\partial_tf+L_Af=0,\ \ f(0,x)=f_0(x),
\end{equation}
%where $x=(x_1,x_2,y)\in \mathbb{T}^2\times \mathbb{R}$ and
where
\begin{equation}\label{eq:1.4}
L_A=-A^{-1}\Delta+y\partial_{x_1}+y^2\partial_{x_2}.
\end{equation}
%\vskip .05in
%Here
In this paper, we study the enhanced dissipation of non-shear flow in \eqref{eq:1.2} by the Equation \eqref{eq:1.3}. By the semigroup theory and the definition of  linear operator \eqref{eq:1.4}, we only to study semigroup with the operator $-L_A$ as generator and give the semigroup estimate.

%\newpage
%First, we study the enhanced dissipation of three dimensional non-shear incompressible flow ${\bf u}$ in %\eqref{eq:1.2}, consider the linear equation
%\begin{equation}\label{eq:1.3}
%\partial_tf+L_Af=0,\ \ f(0,x)=f_0(x),
%\end{equation}
%where $x=(x_1,x_2,y)\in \mathbb{T}^2\times \mathbb{R}$ and
%\begin{equation}\label{eq:1.4}
%L_A=-A^{-1}\Delta+y\partial_{x_1}+y^2\partial_{x_2}.
%\end{equation}
\vskip .05in

Let denote $P_0$ and $P_{\neq}$ are  projection operator, which are defined that for any $f(x_1,x_2,y)$
\begin{equation}\label{eq:1.5}
f^0=P_0f=\frac{1}{|\mathbb{T}^2|}\int_{\mathbb{T}^2}f(t,x_1,x_2,y)dx_1dx_2,\ \ \ \ f_{\neq}=P_{\neq}f=f-f^0.
\end{equation}
The $e^{-tL_A}$ is a semigroup with the operator $-L_A$ as generator, then we have the following result.
\begin{theorem}\label{thm:1.0}
Let $L_A$ be defined in \eqref{eq:1.4}, then for any $t\geq 0$, one has
\begin{equation}\label{eq:1.6}
\|e^{-tL_A}P_{\neq}\|_{L^2\rightarrow L^2}\leq e^{-\lambda_At+\pi/2},\ \ \ \ \lambda_A=\epsilon_0 A^{-1/2},
\end{equation}
where $P_{\neq}$ is defined in \eqref{eq:1.5}, $\epsilon_0$ is small enough and $A$ is large.
\end{theorem}

\vskip .05in

In the Theorem \ref{thm:1.0},  the enhanced dissipation rate of non-shear flow in \eqref{eq:1.2} is the same as that of the Poiseuille flow ${\bf u}=(y^2,0,0)$, see \cite{Coti.2020}. However, the conditions of enhanced dissipation are different, the enhanced dissipation of shear flow occurs in the nonzero mode of $x_1$, while the enhanced dissipation of non-shear flow in \eqref{eq:1.2} occurs in the nonzero mode of $x_1$ or $x_2$ by the semigroup estimate in \eqref{eq:1.6} and the definition of $P_{\neq}$ in \eqref{eq:1.5}. Otherwise, the non-shear flow in \eqref{eq:1.2} and shear flow have different forms of movement, the velocity of shear flow changes and the direction remains unchanged if $y$ changes, while the velocity and direction of the non-shear flow in \eqref{eq:1.2} both change. Thus, it is a interest problem for considering the enhanced dissipation of the non-shear flow in \eqref{eq:1.2}.%\eqref{eq:001} if ${\bf u}$ is non-shear flow in \eqref{eq:1.2}.

\vskip .05in

Next, we study the enhanced dissipation effect of non-shear flow in nonlinear equations. In this paper, we consider three dimensional parabolic-elliptic Keller-Segel equation with the non-shear flow in \eqref{eq:1.2}. We hope that the enhanced dissipation effect of non-shear flow in \eqref{eq:1.2} can suppress the blow-up of Keller-Segel equation and establish the global classical solution with large data.

\vskip .05in
Now the main result of this paper read as follows.

\begin{theorem}\label{thm:1.1}
Let ${\bf u}$ be a non-shear flow in \eqref{eq:1.2} and initial data $n_0\geq0, n_0\in H^s(\mathbb{T}^2\times \mathbb{R}), s>\frac{7}{2}$, there exists $A_0=A(n_0)$, such that if $A>A_0$, the unique classical solution $n(t,x)$ of Equation \eqref{eq:1.1} is global in time.
\end{theorem}

\vskip .05in

It is well known that the shear flow can suppress blow-up for the solution of three dimensional parabolic-elliptic Keller-Segel equation in the case of the initial mass is less than $8\pi$, see \cite{Bedrossian.2017}. Theorem \ref{thm:1.1} tell us that the non-shear flow in \eqref{eq:1.2} can suppress the blow up of three dimensional parabolic-elliptic Keller-Segel equation for any initial mass. The main reason is that the zero mode equation of three dimensional Keller-Segel equation with shear flow is two dimensional equation, while for three dimensional Keller-Segel equation with non shear flow in \eqref{eq:1.2}, we deduce by the Theorem \ref{thm:1.0} that the zero mode equation is one dimension. It is worth noting that the non-shear flow in \eqref{eq:1.2} is a stationary solution to the three dimensional Navier-Stokes equation, thus we will consider the stability of such flows in three dimensional incompressible Navier-Stokes equation and the suppression of blow-up for the three dimensional Keller-Segel-Navier-Stokes equation in forthcoming papers.

\vskip .05in

In the following, we briefly state our main ideas of the proof. Firstly, we need to prove the enhanced dissipation of non-shear flow in \eqref{eq:1.2} by resolvent estimate and semigroup estimate, see Theorem \ref{thm:1.0}. In the proof, we consider a generalized non-shear flow, see \eqref{eq:3.2}, and study the generalized linear operator $\mathcal{L}_A$, see \eqref{eq:3.3}. If denote $e^{-t\mathcal{L}_A}$ is a semigroup with the operator $-\mathcal{L}_A$ as generator, then we obtain semigroup estimate by resolvent estimate and the Gearchart-Pr\"{u}ss type theorem, see Lemma \ref{lem:2.1}. If $u_0=1$ in \eqref{eq:3.2}, the operator $\mathcal{L}_A$ is reduced to the $L_A$, thus, we can finish the proof of Theorem \ref{thm:1.0}. Next, we prove that the non-shear flow in \eqref{eq:1.2} can suppress the blow-up for three dimensional parabolic-elliptic Keller-Segel equation and establish global classical solution, see Theorem \ref{thm:1.1}. Since $L^2$ is supercritical estimate and local existence of solution to Equation \eqref{eq:1.1} is standard, thus we only need to establish global $L^2$ estimate. We decompose Equation \eqref{eq:1.1} into one-dimensional zero mode equation and three-dimensional nonzero mode equation, see \eqref{eq:2.2} and \eqref{eq:2.3}. Since one dimensional Keller-Segel equation is global existence and enhanced dissipation effect of non-shear flow in the three dimensional nonzero mode equations, we can establish global $L^2$ estimate by bootstrap argument.

\vskip .05in

In this paper, we study the enhanced dissipation of non-shear flow in \eqref{eq:1.2} and blow-up suppression in the three dimensional Keller-Segel equation, some proof techniques and ideas are inspired by \cite{Bedrossian.2017} and \cite{Feng.2022}. We know that the resolvent estimate method is widely applied to study the  enhanced dissipation of shear flow. In this paper, we use resolvent estimate method to study the enhanced dissipation of non-shear flow in \eqref{eq:1.2}. In fact, we need to study the linear operator $L_A$ in \eqref{eq:1.4}, our strategy is to transform operator $L_A$ into the shear flow case through transformation and calculation,  which makes the analysis more complex and difficult. In addition, since the enhanced dissipation effect of non-shear flow in \eqref{eq:1.2} occurs in the nonzero mode of $x_1$ or $x_2$, thus we need to study different situations in the analysis, the details can refer to Section \ref{sec.3}. We study that the non-shear flow suppress blow-up of three dimensional Keller-Segel equation by bootstrap argument, where we only consider the nonzero mode equation in the bootstrap argument, see Assumption \ref{assm:2.2} and Proposition \ref{prop:2.3}. And the estimate of Equation \eqref{eq:2.1} and zero mode equation \eqref{eq:2.2} are also needed, see Lemma \ref{lem:4.1}-\ref{lem:4.3}. Some mathematical techniques are used in the proof, such as energy methods, Moser iteration, semigroup estimation and Duhamel's principle, the details can refer to Section \ref{sec.4}.

\vskip .05in

%The paper is organized as follows.
%The rest of this paper is arranged as follows.
%In Section \ref{sec.1}, we introduce some background and our main results.

The rest of this paper is arranged as follows. In Section \ref{sec.2}, we introduce some preparations and give the bootstrap argument. In Section \ref{sec.3}, we prove the enhanced dissipation of non-shear flow in \eqref{eq:1.2}. In Section \ref{sec.4}, we prove  the Theorem \ref{thm:1.1} to establish the global well-posedness of three dimensional Keller-Segel equation with non-shear flow in \eqref{eq:1.2}. In the appendix, we provide necessary supplements to this paper.

\vskip .05in

Throughout the paper, we use the standard notations to denote function spaces and use $C$ to denote a generic constant which may vary from line to line.

\section{Preliminaries and Bootstrap argument}\label{sec.2}
In what follows, we provide some  notations and the auxiliary results, and set up the bootstrap argument. The details are as follows.
\subsection{Notations}
For quantities $X,Y$, if there exist a constant $C$ such that $X\leq CY$, we write $X\lesssim Y$. In this paper, we study the enhanced dissipation of incompressible flow by resolvent estimate, thus we need to introduce some the operator theory. Let $(\mathcal{X},\|\cdot\|)$ be a complex Hilbert space and let $H$ be a closed linear operator in $\mathcal{X}$ with domain $D(H)$. $H$ is m-accretive if the left open half-plane is contained in the resolvent set with
$$
(H+\lambda I)^{-1}\in \mathcal{B}(\mathcal{X})\,,\quad \|(H+\lambda I)^{-1}\|\leq (Re \lambda)^{-1},\ \  Re\lambda >0,
$$
where $\mathcal{B}(\mathcal{X})$ denotes the set of bounded linear operators on $\mathcal{X}$ with operator norm $\|\cdot\|$ and $I$ is the identity operator.

\vskip .05in

We denote $e^{-tH}$ is a  semigroup with $-H$ as generator and define
\begin{equation}\label{eq:2.1.1}
\Psi(H)=\inf\{\|(H-i\lambda I)f\|: f\in D(H), \lambda\in \mathbb{R}, \|f\|=1 \}.
\end{equation}
The following result is the Gearchart-Pr\"{u}ss type theorem for m-accretive operators, see \cite{Wei.2021}.
\begin{lemma}\label{lem:2.1}
Let $H$ be an m-accretive operator in a Hilbert space $\mathcal{X}$. Then for any $t\geq 0$, we have
$$
\|e^{-tH}\|_{L^2\rightarrow L^2}\leq e^{-t\Psi(H)+\pi/2},
$$
where $\Psi(H)$ is defined in \eqref{eq:2.1.1}.
\end{lemma}

\subsection{Some preparations for equations}
We study the enhanced dissipation effect of Equation \eqref{eq:1.1} by the operator $L_A$ in \eqref{eq:1.4} and Theorem \ref{thm:1.0}, thus we modify the Equation \eqref{eq:1.1} by time scale. If we take $t=A\tau$ and the ${\bf u}$ is non-shear incompressible flow in \eqref{eq:1.2}, then the Equation (\ref{eq:1.1}) can be changed to
\begin{equation}\label{eq:2.1}
\begin{cases}
\partial_tn+y \partial_{x_1}n+ y^2 \partial_{x_2}n-\frac{1}{A}\Delta n+\frac{1}{A}\nabla\cdot(n \nabla c)=0,\\
-\Delta c=n-c, \\
n(0,x_1,x_2,y)=n_0(x_1,x_2,y),\ \ \ (x_1,x_2,y)\in \mathbb{T}^2\times \mathbb{R}.
\end{cases}
\end{equation}
%where $\Omega=\mathbb{T}^2\times \mathbb{R}$.
Since the equivalence of Equations \eqref{eq:1.1} and \eqref{eq:2.1} by time scale transformation, thus we only consider the Equations \eqref{eq:2.1} and establish the global classical solution in this paper.

\vskip .05in

The local well-posedness and the $L^2$-criterion of solution to Equation \eqref{eq:2.1} are established.  It is as follows.

\begin{proposition}\label{prop:2.1}
Let initial data $n_0\geq0, n_0(x_1,x_2,y)\in H^s(\mathbb{T}^2\times \mathbb{R}), s>\frac{7}{2}$, there exists a time $T_\ast=T(n_0)>0$ such that the non-negative solution of Equation \eqref{eq:2.1}
$$
n(t,x_1,x_2,y)\in C([0,T_\ast], H^s(\mathbb{T}^2\times \mathbb{R})).
$$
Moreover, if for a given $T$, the solution verifies the following bound
$$
\lim_{t\rightarrow T}\sup_{0\leq \tau\leq t}\|n(\tau,\cdot)\|_{L^2}<\infty,
$$
then it may be extended up to time $T+\delta$ for small enough $\delta>0$. Furthermore, if $n_0\in L^1(\mathbb{T}^2\times \mathbb{R})$, then the $L^1$ norm of the solution to Equation \eqref{eq:2.1} is preserved for all time, namely
$$
M=\|n\|_{L^1}=\|n_0\|_{L^1}.
$$
\end{proposition}
The Proposition \ref{prop:2.1} tell us that to get the classical solution of Equation \eqref{eq:2.1}, only need to have certain control of spatial $L^2$ norm of the solution. The proofs of local well-posedness and $L^1$ norm conservation is standard method and the proof continuation criterion can refer to \cite{Hopf.2018,Kiselev.2016}.

\subsection{Bootstrap argument}We know that the enhanced dissipation of non-shear flow in \eqref{eq:1.2} occurs in the part where the $x_1$ and $x_2$ directions are not all zero frequency. Similar to paper \cite{Feng.2022,He.2018}, denote
$$
n^{0}=P_0n,\ \  n_{\neq}=P_{\neq}n,\ \ c^{0}=P_0c,\ \  c_{\neq}=P_{\neq}c,
$$
where the $P_0$ and $P_{\neq}$ are defined in \eqref{eq:1.5}. Decompose the solution of Equation \eqref{eq:2.1} into $x_1,x_2$-independent part and $x_1,x_2$-dependent part. Then we obtain the one dimensional zero mode equation
\begin{equation}\label{eq:2.2}
\begin{cases}
\partial_tn^{0}-\frac{1}{A}\partial_{yy}n^{0}+\frac{1}{A}\partial_y(n^{0} \partial_y c^0)+\frac{1}{A}\left(\nabla\cdot(n_{\neq} \nabla c_{\neq})\right)^{0}=0,\\
-\partial_{yy} c^0=n^0-c^0,
\end{cases}
\end{equation}
and three dimensional nonzero mode equation
\begin{equation}\label{eq:2.3}
\begin{cases}
\partial_tn_{\neq}+ y \partial_{x_1}n_{\neq}+ y^2 \partial_{x_2}n_{\neq}-\frac{1}{A}\Delta n_{\neq}+\frac{1}{A}\nabla n^{0}\cdot\nabla c_{\neq}+\frac{1}{A}\nabla n_{\neq}\cdot\nabla c^{0}\\
\ \ \ \ \ \ \ \ \ \ \ \ \ \ \ \ \ \  \ \  +\frac{1}{A}\left(\nabla\cdot(n_{\neq} \nabla c_{\neq})\right)_{\neq}-\frac{1}{A}n^{0}(n_{\neq}-c_{\neq})-\frac{1}{A} n_{\neq}(n^0-c^0)=0,\\
-\Delta c_{\neq}=n_{\neq}-c_{\neq}.
\end{cases}
\end{equation}

\vskip .05in

In this paper, we establish the global well-posedness based on the standard bootstrap argument. We list the bootstrap assumptions as below.

\begin{assumption}\label{assm:2.2}
Let $n_{\neq}$ be the solution of Equation \eqref{eq:2.3} with $n_0$ and the constants $C_{\neq}$ is positive constant, which are determined by the proof. Define $T^\ast=T(n_0)>0$ to be the maximum time such the following assumption hold,
\begin{itemize}
\item [(A-1)] Nonzero mode $L^2\dot{H}^1$ estimate: for any $0\leq s\leq t\leq T^\ast$
$$
\frac{1}{A}\int_{s}^{t}\|\nabla n_{\neq}\|^2_{L^2}d\tau\leq 16C_{\neq}e^{-\lambda_A s}\|n_{0}\|^2_{L^2},
$$
\item [(A-2)]Nonzero mode enhanced dissipation estimate: for any $0\leq t\leq T^\ast$
$$
\|n_{\neq}(t)\|^2_{L^2}\leq 4C_{\neq}e^{-\lambda_A t}\|n_{0}\|^2_{L^2},
$$
\end{itemize}
where $\lambda_A$ is defined in \eqref{eq:1.6}.
\end{assumption}

We aim to show $T^\ast=\infty$, this is achieved through the bootstrap argument. To be specific, we will prove the following refined estimates hold on $[0, T^\ast]$ by choosing proper $A$.

\begin{proposition}\label{prop:2.3}
Let $n_{\neq}$ be the solution of Equation \eqref{eq:2.3} with $n_0$ and satisfy the Assumption {\rm \ref{assm:2.2}}. Then there exist $A_0=A(n_0)$, such that $A>A_0$, one has
\begin{itemize}
\item [(P-1)] Nonzero mode $L^2\dot{H}^1$ estimate: for any $0\leq s\leq t\leq T^\ast$
$$
\frac{1}{A}\int_{s}^{t}\|\nabla n_{\neq}\|^2_{L^2}d\tau\leq 8C_{\neq}e^{-\lambda_A s}\|n_{0}\|^2_{L^2},
$$
\item [(P-2)]Nonzero mode enhanced dissipation estimate: for any $0\leq t\leq T^\ast$
$$
\|n_{\neq}(t)\|^2_{L^2}\leq 2C_{\neq}e^{-\lambda_A t}\|n_{0}\|^2_{L^2},
$$
\end{itemize}
where $\lambda_A$ is defined in \eqref{eq:1.6}.
\end{proposition}

We know that the time $T^\ast$ in Assumption \ref{assm:2.2} is large than $8\lambda^{-1}_{A}$ by local estimate, the details can see Appendix. The $L^\infty L_y^2$ estimate of $n^0$, $L^\infty L^\infty$ estimate of $n$ and $L^\infty \dot{H}^1$ estimate of $n^0$ can be obtained by the Assumption \ref{assm:2.2}, see Section \ref{sec.4}. Combining the Assumption \ref{assm:2.2} and Proposition \ref{prop:2.3}, we imply that the $T^\ast$ is infinity by bootstrap argument, then the global $L^2$ estimate of $n$ is established. Based on the local solution and uniform $L^2$ estimate of $n$, the global classical solution of Equation \eqref{eq:2.1} can be established by Proposition \ref{prop:2.1}. Since the equivalence of Equations \eqref{eq:1.1} and \eqref{eq:2.1} by time scale transformation, Then we can finish the proof of Theorem \ref{thm:1.1}. In this paper, the Proposition \ref{prop:2.3} is the most important and we can prove it by enhanced dissipation effect of non-shear flow in \eqref{eq:1.2}, the details of proof can be seen in Section \ref{sec.4}.

\vskip .1in

\section{Enhanced dissipation of non-shear incompressible flow}\label{sec.3}

In this section, we study the enhanced dissipation effect of non-shear flow in \eqref{eq:1.2} and finish the proof of Theorem \ref{thm:1.0}. We consider the linear advection diffusion equation
\begin{equation}\label{eq:3.1}
\partial_t f +{\bf u}\cdot\nabla f -A^{-1}\Delta f=0,\ \ \ f(0,x)=f_0(x),
\end{equation}
where consider  the generalized non-shear incompressible flow
\begin{equation}\label{eq:3.2}
{\bf u}=(u_0(y)y,u_0(y)y^2,0).
\end{equation}
Similar to \eqref{eq:1.3} and \eqref{eq:1.4}, we know that the key point of enhanced dissipation to Equation \eqref{eq:3.1} is to study the linear operator
\begin{equation}\label{eq:3.3}
\mathcal{L}_A=-A^{-1}\Delta+u_0(y)y\partial_{x_1}+u_0(y)y^2\partial_{x_2}.
\end{equation}
%To study the enhanced dissipation of ${\bf u}$ in \eqref{eq:3.2},
For the non-shear flow in \eqref{eq:3.2}, we give the following assumption on $u_0(y)$, which is inspired by  the previous work of one of the authors in \cite{Feng.2022}.
\begin{assumption}\label{ass:3.1}
Let $l$ be a positive integer, there exist $m, N\in\mathbb{N},c_1>0,\delta_0>0$ with the property that, for
any $\lambda , \alpha_1,  \alpha_2\in \mathbb{R}$ and $\delta\in(0,\delta_0)$, there exist finitely many  points $y_1,\ldots y_n\in \mathbb{R}$ with $n\leq N$, such that
$$
\left|u_0(y)\left((y+\alpha_1)^l+\alpha_2\right)-\lambda\right|\geq c_1 \delta^{m+l}, \ \ \ \forall \  |y-y_j|\geq \delta,  \ \ \ \forall j\in \{1,\ldots n\}.
$$
\end{assumption}

\vskip .05in

If denote $e^{-t\mathcal{L}_A}$ is a semigroup with the operator $-\mathcal{L}_A$ as generator and based on such an assumption on $u_0(y)$, we state a proposition as follows.
\begin{proposition}\label{prop:3.2}
Let $u_0(y)$ satisfy Assumption {\rm \ref{ass:3.1}}, then there exist a positive constant $\epsilon_0$ is independent of $A$, such that for any $t\geq 0$, one has
$$
\|e^{-t\mathcal{L}_A}P_{\neq}\|_{L^2\rightarrow L^2}\leq e^{-\lambda'_At+\pi/2},\ \ \ \ \lambda'_A=\epsilon_0 A^{-\frac{m+2}{m+4}},
$$
where $P_{\neq}$ is defined in~\eqref{eq:1.5} and the operator $\mathcal{L}_A$ is defined in \eqref{eq:3.3}.
\end{proposition}

%\vskip .05in

We need to provide some explanations for the Assumption \ref{ass:3.1} and Proposition \ref{prop:3.2}. The assumption is reasonable and some special cases can be pointed out here. If $u_0(y)=1$, then $m=0$, the Assumption \ref{ass:3.1} is trivial by Taylor expansion. In the case of $u_0(y)=1$, if $l=1, \alpha_1=\alpha_1=0$, the assumption is similar with the Couette flow, if $l=2, \alpha_1=\alpha_2=0$, the assumption is similar with the Poiseuille flow. And we know that the operator $\mathcal{L}_A=L_A$ in the case of $u_0(y)=1$, thus we can finish the proof of Theorem \ref{thm:1.0} by the Proposition \ref{prop:3.2}.

\vskip .05in

Next, we prove the Proposition \ref{prop:3.2}. If $f(t,x_1,x_2,y)$  is the solution of Equation \eqref{eq:3.1}, then taking the Fourier transform in $x_1$ and $x_2 $, it is as follows
\begin{equation}\label{eq:3.0}
f_{k}(t,y)=\frac{1}{|\mathbb{T}^2|}\int_{\mathbb{T}^2}f(t,x_1,x_2,y)e^{-i(k_1 x_1+k_2 x_2)}dx_1dx_2,\ \ \ k=(k_1,k_2)\neq {\bf 0}.
\end{equation}
Here we denote
$$
k=(k_1,k_2)\neq {\bf 0},\ \ \ k^2=k_1^2+k_2^2,\ \ |k|=
\begin{cases}
|k_1|,\ \ \ k_2=0,\\
|k_2|,\ \ \ k_2\neq0.
\end{cases}
$$
Thus $f_{k}(t,y)$ satisfy
\begin{equation}\label{eq:3.01}
\partial_tf_k+\mathcal{L}_{A,k}f_k=0,\ \ \ f_k(0,y)=f_{0,k}(y),
\end{equation}
and
\begin{equation}\label{eq:3.4}
\mathcal{L}_{A,k}=A^{-1}(k^2-\partial_{yy})+ik_1u_0(y) y +ik_2u_0(y)y^2.
\end{equation}
We study the semigroup with the operator $-\mathcal{L}_{A,k}$ as generator and establish a lower bound of $\Psi(\mathcal{L}_{A,k})$, which is defined by \eqref{eq:2.1.1} and \eqref{eq:3.4}, it is as follows
\begin{equation}\label{eq:3.5}
\Psi(\mathcal{L}_{A,k})=\inf\{\|(\mathcal{L}_{A,k}-i\lambda I)f\|: f\in D(\mathcal{L}_{A,k}), \lambda\in \mathbb{R}, \|f\|=1 \}.
\end{equation}
Similar to the argument in \cite{Wei.2021}, we know that for any $f\in D(\mathcal{L}_{A,k})=H^{2}(\mathbb{R})$, one has
$$
Re\langle \mathcal{L}_{A,k} f, f\rangle \geq 0,
$$
and
$$
\|(\mathcal{L}_{A,k}+\lambda I)f\|_{L^2}\|f\|_{L^2}\geq Re\langle (\mathcal{L}_{A,k}+\lambda I)f, f\rangle\geq (Re\lambda)\|f\|_{L^2}^2,
$$
therefore, the $\mathcal{L}_{A,k}$ is $m$-accretive operator. Base on the Assumption \ref{ass:3.1}, we have the following the Gearchart-Pr\"{u}ss type theorem and the lower bound of $\Psi(\mathcal{L}_{A,k})$ in \eqref{eq:3.5}.

\begin{lemma}\label{lem:3.3}
Let $\mathcal{L}_{A,k}$ be an $m$-accretive operator in \eqref{eq:3.4}, then one has
\begin{equation}\label{eq:3.6}
\|e^{-t\mathcal{L}_{A,k}}\|_{L^2\rightarrow L^2}\leq e^{-t\Psi(\mathcal{L}_{A,k})+\pi/2}.
\end{equation}
If $u_0(y)$ satisfy the Assumption {\rm \ref{ass:3.1}}, $k\neq(0,0)$ and $A^{-1}|k|^{-1}<1$. Then there exist a positive constant $\epsilon_0$ is independent of $A$ and $k$, such that
\begin{equation}\label{eq:3.7}
\Psi(\mathcal{L}_{A,k})\geq \epsilon_0 A^{-\frac{m+2}{m+4}}|k|^{\frac{2}{m+4}},
\end{equation}
where $\Psi(\mathcal{L}_{A,k})$ is defined in \eqref{eq:3.5}.
\end{lemma}

\begin{proof}
Since $\mathcal{L}_{A,k}$ is an $m$-accretive operator, then the \eqref{eq:3.6} is trivial by Lemma \ref{lem:2.1}. Next, we will prove the \eqref{eq:3.7}. We consider the following two cases.

\vskip .05in

\noindent {\bf Case 1}.\ \ If $k_2=0$, the operator $\mathcal{L}_{A,k}$ can be written as
$$
\mathcal{L}_{A,k}=A^{-1}(k^2-\partial_{yy})+ik_1u_0(y) y,
$$
and since $k=(k_1,k_2)\neq {\bf 0}$, we know that $k_1\neq 0$ and $|k|=|k_1|$. For any fixed $\lambda\in \mathbb{R}$, define
\begin{equation}\label{eq:3.8}
\widetilde{\mathcal{L}}_{A,k}=\mathcal{L}_{A,k}-i\lambda=A^{-1}(k^2-\partial_{yy})+ik_1\left(u_0(y)y-\widetilde{\lambda}\right),
\end{equation}
where $\widetilde{\lambda}=\lambda/k_1 $. Taking the set as follows
\begin{equation}\label{eq:3.9}
E=\{y\in \mathbb{R}: |y-y_j|\geq \delta, \quad \forall j\in\{1,\ldots, n\} \},
\end{equation}
and the $y_j$ satisfy the Assumption \ref{ass:3.1}. Since $u_0(y)$ is continuous function, we define function
$$
\chi: \mathbb{R}\rightarrow [-1,1]
$$
is a smooth approximation function of $  sign\left(u_0(y)y-\widetilde{\lambda} \right)$, there exist $c_2>0$, such that for any $y\in \mathbb{R}$, one has
\begin{equation}\label{eq:3.10}
|\chi'(y)|\leq c_2\delta^{-1},  \ \ \ |\chi''(y)|\leq c_2\delta^{-2},
\end{equation}
and
\begin{equation}\label{eq:3.11}
 \chi(y) \left(u_0(y)y-\widetilde{\lambda} \right)\geq 0.
\end{equation}
In addition, for any $y\in E$, one has
\begin{equation}\label{eq:3.12}
\chi(y)\left(u_0(y)y-\widetilde{\lambda} \right)=|u_0(y)y-\widetilde{\lambda}| .
\end{equation}
For any $f\in D(\mathcal{\widetilde{L}}_{A,k})$ and $\|f\|_{L^2}=1$, we obtain by the definition of $\mathcal{\widetilde{L}}_{A,k}$ in \eqref{eq:3.8} that
$$
\begin{aligned}
\langle\mathcal{\widetilde{L}}_{A,k}f, \chi f \rangle &=A^{-1}\langle(k^2-\partial_{yy})f,  \chi f\rangle+ik_1\langle (u_0(y)y-\widetilde{\lambda})f,\chi f\rangle\\
&=-A^{-1}\langle\partial_{yy}f, \chi f\rangle+A^{-1} k^2\langle f, \chi f\rangle+ik_1\langle (u_0(y)y-\widetilde{\lambda})f,\chi f\rangle\\
&=A^{-1}\langle\partial_{y}f, \chi' f\rangle+A^{-1}\langle\partial_{y}f, \chi \partial_{y}f\rangle+A^{-1} k^2\langle f, \chi f\rangle+ik_1\langle (u_0(y)y-\widetilde{\lambda})f,\chi f\rangle,
\end{aligned}
$$
then
\begin{equation}\label{eq:3.13}
{\rm Im}\langle\mathcal{\widetilde{L}}_{A,k}f, \chi f \rangle=A^{-1}{\rm Im}\langle\partial_{y}f, \chi' f\rangle+k_1\langle (u_0(y)y-\widetilde{\lambda})f,\chi f\rangle.
\end{equation}
Since $u_0(y)$ satisfy the Assumption \ref{ass:3.1}, one gets
$$
\left|u_0(y)y-\widetilde{\lambda}\right|\geq c_1 \delta^{m+1}, \ \  y\in E,
$$
then we deduce by \eqref{eq:3.11} and \eqref{eq:3.12} that
\begin{equation}\label{eq:3.14}
\langle (u_0(y)y-\widetilde{\lambda})f,\chi f\rangle\geq \int_{E}|u_0(y)y-\widetilde{\lambda}||f(y)|^2dy\geq c_1 \delta^{m+1}\int_{E}|f(y)|^2dy.
\end{equation}
By \eqref{eq:3.10}, \eqref{eq:3.13} and Cauchy-Schwartz inequality, one has
\begin{equation}\label{eq:3.15}
\begin{aligned}
|k|\langle (u_0(y)y-\widetilde{\lambda})f,\chi f\rangle&=\left|k_1\langle (u_0(y)y-\widetilde{\lambda})f,\chi f\rangle\right|\\
&=\left|{\rm Im}\langle\mathcal{\widetilde{L}}_{A,k}f, \chi f \rangle-A^{-1}{\rm Im}\langle\partial_{y}f, \chi' f\rangle\right|\\
&\leq \|\mathcal{\widetilde{L}}_{A,k}f\|_{L^2}\|\chi f\|_{L^2}+A^{-1} \|\partial_{y}f\|_{L^2}\|\chi' f\|_{L^2} \\
&\leq \|\mathcal{\widetilde{L}}_{A,k}f\|_{L^2}\|f\|_{L^2}+c_2A^{-1} \delta^{-1}\|\partial_{y}f\|_{L^2}\|f\|_{L^2}.
\end{aligned}
\end{equation}
Combining \eqref{eq:3.14} and \eqref{eq:3.15}, we obtain
\begin{equation}\label{eq:3.16}
\begin{aligned}
\int_{E}|f(y)|^2dy&\leq c_1^{-1}\delta^{-m-1}\langle (u_0(y)y-\widetilde{\lambda})f,\chi f\rangle\\
&\leq c_1^{-1}\delta^{-m-1}|k|^{-1}\left(\|\mathcal{\widetilde{L}}_{A,k}f\|_{L^2}\|f\|_{L^2}+c_2A^{-1} \delta^{-1}\|\partial_{y}f\|_{L^2}\|f\|_{L^2}\right).
\end{aligned}
\end{equation}
Since
$$
\langle \mathcal{\widetilde{L}}_{A,k}f, f \rangle=A^{-1}\langle (k^2-\partial_{yy})f,f \rangle+ik_1\left\langle \left(u_0(y)y-\widetilde{\lambda}\right)f,f \right\rangle,
$$
we can easily get
$$
{\rm Re} \langle \mathcal{\widetilde{L}}_{A,k}f, f \rangle=A^{-1} k^2\|f\|_{L^2}^2+A^{-1} \|\partial_{y}f\|_{L^2}^2.
$$
Thus, we deduce by Cauchy-Schwartz inequality that
\begin{equation}\label{eq:3.17}
\|\partial_{y}f\|_{L^2}\leq A^{\frac{1}{2}}\left({\rm Re} \langle \mathcal{\widetilde{L}}_{A,k}f, f \rangle \right)^{\frac{1}{2}}
\leq A^{\frac{1}{2}}\|\mathcal{\widetilde{L}}_{A,k}f\|_{L^2}^{\frac{1}{2}}\|f\|_{L^2}^{\frac{1}{2}}.
\end{equation}
Combining \eqref{eq:3.16} and \eqref{eq:3.17}, one gets
$$
\int_{E}|f(y)|^2dy\leq c_1^{-1}\delta^{-m-1}|k|^{-1}\left(\|\mathcal{\widetilde{L}}_{A,k}f\|_{L^2}\|f\|_{L^2}+c_2A^{-\frac{1}{2}} \delta^{-1}\|\mathcal{\widetilde{L}}_{A,k}f\|_{L^2}^{\frac{1}{2}}\|f\|_{L^2}^{\frac{3}{2}}\right).
$$
By using Young's inequality, to obtain
$$
c_1^{-1}\delta^{-m-1}|k|^{-1}c_2A^{-\frac{1}{2}} \delta^{-1}\|\mathcal{\widetilde{L}}_{A,k}f\|_{L^2}^{\frac{1}{2}}\|f\|_{L^2}^{\frac{3}{2}}\leq \frac{1}{4}\|f\|_{L^2}^2+\widetilde{c}|k|^{-2}A^{-1} \delta^{-2m-4}\|\mathcal{\widetilde{L}}_{A,k}f\|_{L^2}\|f\|_{L^2},
$$
where $\widetilde{c}=c^2_2/c^2_1$. Thus the \eqref{eq:3.16} can be written as
\begin{equation}\label{eq:3.18}
\int_{E}|f(y)|^2dy\leq \left(c_1^{-1}\delta^{-m-1}|k|^{-1}+\widetilde{c}|k|^{-2}A^{-1} \delta^{-2m-4}\right)\|\mathcal{\widetilde{L}}_{A,k}f\|_{L^2}\|f\|_{L^2}+\frac{1}{4}\|f\|_{L^2}^2.
\end{equation}
Denoted $E^c$ as the complement of $E$, then $|E^c|\leq 2N\delta$ by the definition of $E$ in \eqref{eq:3.9} and
$$
\int_{E^c}|f(y)|^2dy\leq 2N\delta \|f\|^2_{L^\infty},
$$
by Gagliardo-Nirenberg inequality, one has
$$
\|f\|_{L^\infty}\lesssim\|f\|^{\frac{1}{2}}_{L^2}\|\partial_yf\|^{\frac{1}{2}}_{L^2}.
$$
Then we deduce by \eqref{eq:3.17} and Young's inequality that
\begin{equation}\label{eq:3.19}
\begin{aligned}
\int_{E^c}|f(y)|^2dy&\lesssim N\delta\|f\|_{L^2}\|\partial_yf\|_{L^2}\\
&\lesssim N\delta A^{\frac{1}{2}}\|\mathcal{\widetilde{L}}_{A,k}f\|_{L^2}^{\frac{1}{2}}\|f\|_{L^2}^{\frac{3}{2}}\\
&\leq \frac{1}{4}\|f\|_{L^2}^2+ C(N\delta)^2A\|\mathcal{\widetilde{L}}_{A,k}f\|_{L^2}\|f\|_{L^2}.
\end{aligned}
\end{equation}
Combining \eqref{eq:3.18} and \eqref{eq:3.19}, we have
\begin{equation}\label{eq:3.20}
\|f\|_{L^2}^2\leq 2\left(c_1^{-1}\delta^{-m-1}|k|^{-1}+\widetilde{c}|k|^{-2}A^{-1} \delta^{-2m-4}+C(N\delta)^2A \right)\|\mathcal{\widetilde{L}}_{A,k}f\|_{L^2}\|f\|_{L^2}.
\end{equation}
Taking $\delta$ small enough and
$$
\delta=c_3\left(|k|A\right)^{-\frac{1}{m+3}},
$$
where $c_3>0$ is small constant. Then there exist a constant $C_0=C(c_1,c_2,c_3,m,N)$, such that
$$
c_1^{-1}\delta^{-m-1}|k|^{-1}+\widetilde{c}|k|^{-2}A^{-1} \delta^{-2m-4}+C(N\delta)^2A \leq C_0A^{\frac{m+1}{m+3}}|k|^{-\frac{2}{m+3}}.
$$
Therefore, we can imply by \eqref{eq:3.20} that
$$
\|\mathcal{\widetilde{L}}_{A,k}f\|_{L^2}\geq \epsilon_0A^{-\frac{m+1}{m+3}}|k|^{\frac{2}{m+3}}\|f\|_{L^2},
$$
where $\epsilon_0=1/2C_0$. Since $f$ is arbitrary, then we have
\begin{equation}\label{eq:3.21}
\Psi(\mathcal{L}_{A,k})\geq \epsilon_0A^{-\frac{m+1}{m+3}}|k|^{\frac{2}{m+3}}.
\end{equation}

\vskip .05in

\noindent {\bf Case 2}.\ \ If $k_2\neq0$, the operator $\mathcal{L}_{A,k}$ can be written as
$$
\begin{aligned}
\mathcal{L}_{A,k}&=A^{-1}(k^2-\partial_{yy})+ik_1u_0(y) y +ik_2u_0(y)y^2\\
&=A^{-1}(k^2-\partial_{yy})+ik_2u_0(y)\left(y^2+k_1k^{-1}_2y\right)\\
&=A^{-1}(k^2-\partial_{yy})+ik_2u_0(y)\left((y+\alpha_{k,1})^2+\alpha_{k,2}\right)
\end{aligned}
$$
where
$$
\alpha_{k,1}=\frac{k_1}{2k_2},\ \ \ \ \alpha_{k,2}=-\frac{k^2_1}{4k^2_2}.
$$
For any fixed $\lambda\in \mathbb{R}$, define
\begin{equation}\label{eq:3.22}
\widetilde{\mathcal{L}}_{A,k}=\mathcal{L}_{A,k}-i\lambda=A^{-1}(k^2-\partial_{yy})+ik_2
\left(u_0(y)\left((y+\alpha_{k,1})^2+\alpha_{k,2}\right)-\widetilde{\lambda}\right),
\end{equation}
where $\widetilde{\lambda}=\lambda/k_2$. Similar to the definition of $E$ in \eqref{eq:3.9} and $u_0(y)$ satisfy the Assumption \ref{ass:3.1}, one has
$$
\left|u_0(y)\left((y+\alpha_{k,1})^2+\alpha_{k,2}\right)-\widetilde{\lambda}\right|\geq c_1\delta^{m+2},\ \ \  y\in E.
$$
Similar to the discussion of Case 1, we obtain
\begin{equation}\label{eq:3.23}
\|f\|_{L^2}^2\leq 2\left(c_1^{-1}\delta^{-m-2}|k|^{-1}+\widetilde{c}|k|^{-2}A^{-1} \delta^{-2m-6}+C(N\delta)^2A \right)\|\mathcal{\widetilde{L}}_{A,k}f\|_{L^2}\|f\|_{L^2}.
\end{equation}
Taking $\delta$ small enough and
$$
\delta=c_3\left(|k|A\right)^{-\frac{1}{m+4}},
$$
where $c_3>0$ is small constant. Then there exist a constant $C_0=C(c_1,c_2,c_3,m,N)$, such that
$$
c_1^{-1}\delta^{-m-2}|k|^{-1}+\widetilde{c}|k|^{-2}A^{-1} \delta^{-2m-6}+C(N\delta)^2A \leq C_0A^{\frac{m+2}{m+4}}|k|^{-\frac{2}{m+4}}.
$$
Therefore, similar to the discussion of \eqref{eq:3.21}, we obtain
\begin{equation}\label{eq:3.24}
\Psi(\mathcal{L}_{A,k})\geq \epsilon_0A^{-\frac{m+2}{m+4}}|k|^{\frac{2}{m+4}},\ \ \ \epsilon_0=1/2C_0.
\end{equation}
Since  $A^{-1}|k|^{-1}<1$, we know that
$$
A^{-\frac{m+1}{m+3}}|k|^{\frac{2}{m+3}}\geq A^{-\frac{m+2}{m+4}}|k|^{\frac{2}{m+4}}.
$$
Combining Case 1 and Case 2, we imply that for any $k=(k_1,k_2)\neq {\bf 0}$, one has
$$
\Psi(\mathcal{L}_{A,k})\geq \epsilon_0 A^{-\frac{m+2}{m+4}}|k|^{\frac{2}{m+4}}.
$$
This completes the proof of Lemma \ref{lem:3.3}.
\end{proof}

Next, we give the proof of Proposition \ref{prop:3.2} base on the Lemma \ref{lem:3.3}.

\begin{proof}[The proof of Proposition {\rm \ref{prop:3.2}}]
For any $g(x)\in L^2(\mathbb{T}^2\times \mathbb{R})$, considering the Equation \eqref{eq:3.1} with initial data $f_0(x)=P_{\neq}g(x)$, then the solution can be written as
$$
f(t,x)=e^{-t\mathcal{L}_{A}}f_0=e^{-t\mathcal{L}_{A}}P_{\neq}g,
$$
and we can obtain $P_0f=0$. Here the operator $\mathcal{L}_{A}$ is defined in \eqref{eq:3.3}, the operator $P_0$ and $P_{\neq}$ are defined in \eqref{eq:1.5}. Through Fourier series, the solution of Equation \eqref{eq:3.1} can be written as
\begin{equation}\label{eq:3.27}
f(t,x_1,x_2,y)=e^{-t\mathcal{L}_{A}}P_{\neq}g=\sum_{k=(k_1,k_2)\neq {\bf 0}}f_k(t,y)e^{i(k_1x_1+k_2x_2)},
\end{equation}
where $f_k(t,y)$ is defined in \eqref{eq:3.0} and is the solution of Equation \eqref{eq:3.01}. Then we have
$$
f_{k}(t,y)=e^{-t\mathcal{L}_{A,k}}f_{0,k},
$$
and the operator $\mathcal{L}_{A,k}$ is defined in \eqref{eq:3.4}. By the Lemma \ref{lem:3.3}, one gets
$$
\|e^{-t\mathcal{L}_{A,k}}\|_{L^2\rightarrow L^2}\leq e^{-\lambda_{A,k}t+\pi/2},
$$
where $\lambda_{A,k}=\epsilon_0 A^{-\frac{m+2}{m+4}}|k|^{\frac{2}{m+4}}$. Then we deduce by Plancherel equality and \eqref{eq:3.27} that
$$
\begin{aligned}
\|f\|^2_{L^2}&=\|e^{-t\mathcal{L}_{A}}P_{\neq}g\|^2_{L^2}=\sum_{k\neq {\bf 0}}\|f_k\|^2_{L^2}=\sum_{k\neq {\bf 0}}\|e^{-t\mathcal{L}_{A,k}}f_{0,k}\|^2_{L^2}\\
&\leq \sum_{k\neq {\bf 0}}e^{-2\lambda_{A,k} t+\pi}\|f_{0,k}\|^2_{L^2}\leq \max_{k\neq 0}e^{-2\lambda_{A,k} t+\pi}\sum_{k\neq {\bf 0}}\|f_{0,k}\|^2_{L^2}\\
&\leq e^{-2\lambda'_{A} t+\pi}\|f_{0}\|^2_{L^2}\leq e^{-2\lambda'_{A} t+\pi}\|g\|^2_{L^2}\\
%&\lesssim e^{-2\lambda'_{A} t}\|g\|^2_{L^2}.
\end{aligned}
$$
Since $g$ is arbitrary, then we have
$$
\|e^{-t\mathcal{L}_A}P_{\neq}\|_{L^2\rightarrow L^2}\leq e^{-\lambda'_At+\pi/2}.
$$
This completes the proof of Proposition \ref{prop:3.2}.
\end{proof}

%\vskip .05in

\section{Global existence of Equation \eqref{eq:1.1}}\label{sec.4}

In this section, we will finish the proof of Theorem \ref{thm:1.1}. We know that the global classical solution of  Equation \eqref{eq:1.1} can be established if the Proposition \ref{prop:2.3} is true by Section \ref{sec.2}. Thus, the aim of this section is to prove the Proposition \ref{prop:2.3}, and we finish the proof by some lemmas. First, we give the following lemma.

\begin{lemma}[ $L^\infty L_y^2$ estimate of $n^0$ ]\label{lem:4.1}
Let $n^0$ and $n_{\neq}$ are the solutions of Equations \eqref{eq:2.2} and \eqref{eq:2.3} with $n_0$, and $n_{\neq}$ satisfy the Assumption {\rm \ref{assm:2.2}}. There exist $A_0=A(n_0)$ and a positive constant $B_1=B(\|n_0\|_{L^2}, M)$, if $A>A_0$, we have
$$
\|n^{0}\|_{L^\infty(0, T^\ast;L_y^2})\leq B_1.
$$
\end{lemma}

\begin{proof}
Let us multiply both sides of \eqref{eq:2.2} by $n^0$ and integrate over $\mathbb{R}$, we obtain that
\begin{equation}\label{eq:4.1}
\frac{1}{2}\frac{d}{dt}\|n^0\|^2_{L^2_y}+\frac{1}{A}\|\partial_y n^0\|^2_{L^2_y}+\frac{1}{A} \int_{\mathbb{R}}\partial_y(n^{0} \partial_y c^0)n^0dy+\frac{1}{A}\int_{\mathbb{R}}\left(\nabla\cdot(n_{\neq} \nabla c_{\neq})\right)^{0}n^0dy=0.
\end{equation}
Using integral by part, H\"{o}lder's inequality, Young's inequality, Gagliardo-Nirenberg inequality, Lemma \ref{lem:A.2} and \ref{lem:A.3}, one gets
$$
\begin{aligned}
\left|\frac{1}{A} \int_{\mathbb{R}}\partial_y(n^{0} \partial_y c^0)n^0dy\right|&\lesssim \frac{1}{A}\|n^{0}\|_{L^2_y}\|\partial_y c^0\|_{L^\infty_y}\|\partial_y n^0\|_{L^2_y}\\
&\leq \frac{1}{4A}\|\partial_y n^0\|^2_{L^2_y}+\frac{C}{A}\|n^{0}\|^4_{L^2_y},
\end{aligned}
$$
and
$$
\begin{aligned}
\left|\frac{1}{A}\int_{\mathbb{R}}\left(\nabla\cdot(n_{\neq} \nabla c_{\neq})\right)^{0}n^0dy\right|
&=\left|\frac{1}{A|\mathbb{T}^2|}\int_{\mathbb{T}^2\times \mathbb{R}}\nabla\cdot(n_{\neq} \nabla c_{\neq})n^0dx_1dx_2dy\right|\\
&\lesssim\frac{1}{A} \|n_{\neq}\|_{L^4}\|\nabla c_{\neq}\|_{L^4}\|\partial_yn^0\|_{L_{y}^2}\\
&\lesssim\frac{1}{A} \|n_{\neq}\|^{\frac{5}{4}}_{L^2}\|\nabla n_{\neq}\|^{\frac{3}{4}}_{L^2}\|\partial_yn^0\|_{L_{y}^2}\\
&\leq \frac{1}{4A}\|\partial_yn^0\|^2_{L_{y}^2}+\frac{1}{64C_{\neq}A}\|\nabla n_{\neq}\|^2_{L^2}+\frac{C}{A}\|n_{\neq}\|^{10}_{L^2}.
\end{aligned}
$$
Combining Gagliardo-Nirenberg inequality and  $\|n^{0}\|_{L^1_y}\leq CM$, we also have
$$
-\|\partial_y n^0\|^2_{L^2_y}\leq -\frac{\|n^{0}\|^6_{L^2_y}}{CM^4}.
$$
Then we deduce by \eqref{eq:4.1} that
\begin{equation}\label{eq:4.2}
\frac{d}{dt}\|n^0\|^2_{L^2_y}\leq-\frac{\|n^{0}\|^4_{L^2_y}}{CAM^4}\left(\|n^{0}\|^2_{L^2_y}-CM^4\right)+\frac{1}{32C_{\neq}A}\|\nabla n_{\neq}\|^2_{L^2}+\frac{C}{A}\|n_{\neq}\|^{10}_{L^2}.
\end{equation}
Define
$$
G(t)=\int_{0}^{t}\frac{1}{32C_{\neq}A}\|\nabla n_{\neq}\|^2_{L^2}+\frac{C}{A}\|n_{\neq}\|^{10}_{L^2}ds,
$$
then for any $t\in [0,T^\ast]$, we deduce by Assumption \ref{assm:2.2} that
$$
\int_{0}^{t}\frac{1}{32C_{\neq}A}\|\nabla n_{\neq}\|^2_{L^2}ds\leq \frac{1}{2}\|n_0\|^2_{L^2},
$$
and
$$
\int_{0}^{t}\frac{C}{A}\|n_{\neq}(s)\|^{10}_{L^2}ds
\lesssim \frac{(4C_{\neq})^5\|n_{0}\|^{8}_{L^2}}{5A\lambda_A}\|n_{0}\|^{2}_{L^2},
$$
where $\lambda_A$ is defined in \eqref{eq:1.6} and
$$
\frac{(4C_{\neq})^5\|n_{0}\|^{8}_{L^2}}{5A\lambda_A}=\frac{(4C_{\neq})^5\|n_{0}\|^{8}_{L^2}}{5\epsilon_0A^{\frac{1}{2}}}\rightarrow 0,\ \ \ A\rightarrow \infty.
$$
Then for any  $t\in [0,T^\ast]$ and $A$ is large enough, one has
\begin{equation}\label{eq:4.3}
0\leq G(t)\leq \|n_0\|^2_{L^2}.
\end{equation}
Combining \eqref{eq:4.2} and  \eqref{eq:4.3}, one has
\begin{equation}\label{eq:4.4}
\frac{d}{dt}\left(\|n^0\|^2_{L^2_y}-G(t)\right)
\leq -\frac{\|n^{0}\|^4_{L^2_y}}{CAM^4}\left(\|n^{0}\|^2_{L^2_y}-G(t)-CM^4\right).
\end{equation}
Then we deduce by \eqref{eq:4.4} that
$$
\|n^0\|^2_{L^2_y}\leq CM^4+3\|n_0\|^2_{L^2}\triangleq B^2_1.
$$
This completes the proof of Lemma \ref{lem:4.1}.
\end{proof}

To improve the Assumptions \ref{assm:2.2}, we are left to complete the $L^\infty L^\infty$ estimate of $n$. This will be achieved by the Moser-Alikakos iteration \cite{Alikakos.1979}.

\begin{lemma}\label{lem:4.2}
Let $n, n^0$ and $n_{\neq}$ are the solutions of Equations \eqref{eq:2.1}-\eqref{eq:2.3} with $n_0$, and satisfy the Assumption {\rm \ref{assm:2.2}}. There exist $A_0=A(n_0)$ and a positive constant $B_2=B(\|n_0\|_{L^2}, M, \|n_0\|_{L^\infty}, B_1,C_{\neq})$, if $A>A_0$, we have
$$
\|n\|_{L^\infty(0, T^\ast;L^\infty})\leq B_2.
$$
\end{lemma}

\begin{proof}
Combining Assumption \ref{assm:2.2} and Lemma \ref{lem:4.1}, we deduce that for any $0\leq t\leq T^\ast$, one has
\begin{equation}\label{eq:4.5}
\|n(t)\|^2_{L^2}\leq CB^2_1+4C_{\neq}\|n_{0}\|^2_{L^2} \triangleq K^2_1.
\end{equation}
Let $p_k=2^{k+1}$ with $k\geq 1$, multiplying both sides of \eqref{eq:2.1} by $n^{p_k-1}$ and integrate over $\mathbb{T}^2\times \mathbb{R}$, one has
\begin{equation}\label{eq:4.6}
\begin{aligned}
\frac{1}{p_k}\frac{d}{dt}\|n\|^{p_k}_{L^{p_k}}&+\frac{4(p_k-1)}{Ap_k^2} \int_{\mathbb{T}^2\times \mathbb{R}}|\nabla n^{\frac{p_k}{2}}|^2dx_1dx_2dy\\
&+\frac{1}{A}\int_{\mathbb{T}^2\times \mathbb{R}}\nabla\cdot(n \nabla c)n^{p_k-1}dx_1dx_2dy=0.
\end{aligned}
\end{equation}
Using integral by part, H\"{o}lder's inequality, Young's inequality, Gagliardo-Nirenberg inequality, Lemma \ref{lem:A.3} and \eqref{eq:4.5}, one gets
$$
\begin{aligned}
\left|\frac{1}{A}\int_{\mathbb{T}^2\times \mathbb{R}}\nabla\cdot(n \nabla c)n^{p_k-1}dx_1dx_2dy\right|&=\left|-\frac{2(p_k-1)}{Ap_k}\int_{\mathbb{T}^2\times \mathbb{R}}n^{\frac{p_k}{2}}\nabla c\cdot \nabla n^{\frac{p_k}{2}}dx_1dx_2dy\right|\\
&\lesssim \frac{K_1}{A}\|n^{\frac{p_k}{2}}\|_{L^{4}}\|\nabla n^{\frac{p_k}{2}}\|_{L^2}\\
&\lesssim \frac{K_1}{A}\|\nabla n^{\frac{p_k}{2}}\|^{\frac{13}{10}}_{L^{2}}\|n^{\frac{p_k}{2}}\|^{\frac{7}{10}}_{L^{1}}\\
&\leq \frac{1}{Ap_k} \|\nabla n^{\frac{p_k}{2}}\|^2_{L^2}+\frac{C(K_1p_k)^{\frac{20}{7}}}{Ap_k}\|n^{\frac{p_k}{2}}\|^{2}_{L^{1}},
\end{aligned}
$$
and
$$
-\frac{4(p_k-1)}{Ap^2_k} \int_{\mathbb{T}^2\times \mathbb{R}}|\nabla n^{\frac{p_k}{2}}|^2dx_1dx_2dy\leq -\frac{2}{Ap_k} \|\nabla n^{\frac{p_k}{2}}\|^2_{L^2}.
$$
By Gagliardo-Nirenberg inequality and Young's inequality, we imply that
$$
-\|\nabla n^{\frac{p_k}{2}}\|^{2}_{L^{2}}\leq -\|n^{\frac{p_k}{2}}\|^{2}_{L^{2}}+C\|n^{\frac{p_k}{2}}\|^{2}_{L^{1}}.
$$
Then we obtain from \eqref{eq:4.6} that
\begin{equation}\label{eq:4.7}
\frac{d}{dt}\|n^{\frac{p_k}{2}}\|^{2}_{L^{2}}\leq -\frac{1}{A}\|n^{\frac{p_k}{2}}\|^{2}_{L^{2}}+\frac{C(K_1p_k)^{\frac{20}{7}}}{A}\|n^{\frac{p_k}{2}}\|^{2}_{L^{1}}.
\end{equation}
Denote
$$
y_k(t)=\|n\|^{p_k}_{L^{p_k}}=\|n^{\frac{p_k}{2}}\|^{2}_{L^{2}},  \ \ \ \|n^{\frac{p_k}{2}}\|^{2}_{L^{1}}=(\|n\|^{p_k-1}_{L^{p_k-1}})^2,\ \ \ y_0(t)=\|n\|^{2}_{L^{2}}\leq  K^2_1,
$$
then we deduce by \eqref{eq:4.7} that
%\begin{equation}\label{eq:4.8}
$$
y'_k(t)\leq -\frac{1}{A}y_k(t)+\frac{C(K_1p_k)^{\frac{20}{7}}}{A}y^2_{k-1}(t).
$$
%\end{equation}
Therefore, we imply that %from \eqref{eq:4.8} that
\begin{equation}\label{eq:4.9}
y_k(t)\leq a_k\max\left\{\sup_{t\geq 0}y^2_{k-1},y_k(0),1\right\},
\end{equation}
where
$$
a_k=CK^{\frac{20}{7}}_1 2^{\frac{20}{7}}2^{\frac{20k}{7}}\geq 1, \ \ \ y_k(0)\lesssim \max\{\|n_0\|_{L^1},\|n_0\|_{L^\infty}\}.
$$
Combining \eqref{eq:4.9} and Moser-Alikakos iteration, we deduce that for any $k\geq1$, one has
$$
\|n\|_{L^{p_k}}\leq \sup_{t\geq 0} y^{\frac{1}{p_k}}_k(0)\leq CK^{\frac{20}{7}}_1 2^{\frac{30}{7}}\max\left\{K^{\frac{1}{2}}_1, \|n_0\|_{L^1}, \|n_0\|_{L^\infty}, 1\right\}\triangleq B_2.
$$
As $k\rightarrow \infty$, we have
$$
\|n\|_{L^{\infty}}\leq  B_2.
$$
This completes the proof of Lemma \ref{lem:4.2}.
\end{proof}

Next, we establish the $L^\infty \dot{H}^1$ estimate of $n^0$.

\vskip .1in

\begin{lemma}[$L^\infty \dot{H}^1$ estimate of $n^0$]\label{lem:4.3}
Let $n^0$ and $n_{\neq}$ are the solutions of Equations \eqref{eq:2.2} and \eqref{eq:2.3} with $n_0$, and satisfy the Assumption {\rm \ref{assm:2.2}}. There exist $A_0=A(n_0)$ and a positive constant $B_3=B(\|\partial_y n_0\|_{L^2},\|n_0\|_{L^2}, M, B_1, B_2, C_{\neq})$, if $A>A_0$, we have
$$
\|\partial_yn^{0}\|_{L^\infty(0, T^\ast;L_y^2})\leq B_3.
$$
\end{lemma}

\begin{proof}
Applying $\partial_y$ to \eqref{eq:2.2}, to obtain
\begin{equation}\label{eq:4.10}
\partial_t\partial_yn^{0}-\frac{1}{A}\partial_{yy}\partial_yn^{0}+\frac{1}{A}\partial_y\partial_y(n^{0} \partial_y c^0)+\frac{1}{A}\partial_y\left(\nabla\cdot(n_{\neq} \nabla c_{\neq})\right)^{0}=0.
\end{equation}
Let us multiply both sides of \eqref{eq:4.10} by $\partial_yn^{0}$ and integrate  over $\mathbb{R}$, to obtain
\begin{equation}\label{eq:4.11}
\begin{aligned}
\frac{1}{2}\frac{d}{dt}\|\partial_yn^{0}\|^2_{L^2_y}+\frac{1}{A}\|\partial_{yy}n^{0}\|^2_{L^2_y}&+\frac{1}{A}
\int_{\mathbb{R}}\partial_y\partial_y(n^{0} \partial_y c^0)\partial_yn^{0}dy\\
&+\frac{1}{A}\int_{\mathbb{R}}\partial_y\left(\nabla\cdot(n_{\neq} \nabla c_{\neq})\right)^{0}\partial_yn^{0}dy=0.
\end{aligned}
\end{equation}
By  H\"{o}lder's inequality, Young's inequality and Lemma \ref{lem:A.2}, one gets
\begin{equation}\label{eq:4.12}
\begin{aligned}
&\left|\frac{1}{A}\int_{\mathbb{R}}\partial_y\partial_y(n^{0} \partial_y c^0)\partial_yn^{0}dy\right|\\
=&\left|-\frac{1}{A}\int_{\mathbb{R}}\partial_yn^{0} \partial_y c^0\partial_{yy}n^{0}dy-\frac{1}{A}\int_{\mathbb{R}}n^{0} \partial_{yy} c^0\partial_{yy}n^{0}dy\right|\\
\lesssim&\frac{1}{A}\|\partial_yn^{0}\|_{L_y^2}\|\partial_y c^0\|_{L_y^\infty}\|\partial_{yy}n^{0}\|_{L_y^2}+\frac{1}{A}\|n^{0}\|_{L_y^4}\|\partial_{yy} c^0\|_{L_y^4}\|\partial_{yy}n^{0}\|_{L_y^2}\\
\leq& \frac{1}{4A}\|\partial_{yy}n^{0}\|^2_{L_y^2}+\frac{C(B^2_1+M^2)}{A}\|\partial_yn^{0}\|^2_{L_y^2}.
\end{aligned}
\end{equation}
Combining Gagliardo-Nirenberg inequality and Lemma \ref{lem:A.3}, one has
$$
\|\left(\partial_yn_{\neq}\partial_y c_{\neq}\right)^{0}\|_{L^2_y}=\left(\int_{\mathbb{R}}\left| \frac{1}{|\mathbb{T}^2|}\int_{\mathbb{T}^2}\partial_yn_{\neq}\partial_y c_{\neq}dx_1dx_2\right|^2dy\right)^{\frac{1}{2}}\lesssim\|\partial_yn_{\neq}\partial_y c_{\neq}\|_{L^2},
$$
$$
\|\left(n_{\neq}\partial_{yy} c_{\neq}\right)^{0}\|_{L^2_y}\lesssim\|n_{\neq}\partial_{yy} c_{\neq}\|_{L^2},
$$
$$
\|\partial_yc_{\neq}\|_{L^\infty}\lesssim\|\nabla c_{\neq}\|^{\frac{1}{7}}_{L^2}\|\nabla^2 c_{\neq}\|^{\frac{6}{7}}_{L^4}
\lesssim\|n_{\neq}\|^{\frac{1}{7}}_{L^2}\|n_{\neq}\|_{L^4}^{\frac{6}{7}}.
$$
Then we deduce that
\begin{equation}\label{eq:4.13}
\begin{aligned}
&\left|\frac{1}{A}\int_{\mathbb{R}}\partial_y\left(\nabla\cdot(n_{\neq} \nabla c_{\neq})\right)^{0}\partial_yn^{0}dy\right|\\
\lesssim&\frac{1}{A}\|\left(\partial_yn_{\neq}\partial_y c_{\neq}\right)^{0}\|_{L^2_y}\|\partial_{yy}n^{0}\|_{L^2_y}+\frac{1}{A}\|\left(n_{\neq}\partial_{yy} c_{\neq}\right)^{0}\|_{L^2_y}\|\partial_{yy}n^{0}\|_{L^2_y}\\
\leq&\frac{1}{4A}\|\partial_{yy}n^{0}\|^2_{L^2_y}+ \frac{C}{A}\|\partial_yn_{\neq}\|^2_{L^2}\|n_{\neq}\|^{\frac{2}{7}}_{L^2}\|n_{\neq}\|_{L^4}^{\frac{12}{7}}
+\frac{C}{A}\|n_{\neq}\|^2_{L^\infty}\|n_{\neq}\|^2_{L^2}.
\end{aligned}
\end{equation}
Combining  \eqref{eq:4.11}-\eqref{eq:4.13}, we have
\begin{equation}\label{eq:4.14}
\begin{aligned}
\frac{d}{dt}\|\partial_yn^{0}\|^2_{L^2_y}\leq &-\frac{1}{A}\|\partial_{yy}n^{0}\|^2_{L^2_y}+\frac{C(B^2_1+M^2)}{A}\|\partial_yn^{0}\|^2_{L_y^2}\\
&+\frac{C}{A}\|\partial_yn_{\neq}\|^2_{L^2}\|n_{\neq}\|^{\frac{2}{7}}_{L^2}\|n_{\neq}
\|_{L^4}^{\frac{12}{7}}+\frac{C}{A}\|n_{\neq}\|^2_{L^\infty}\|n_{\neq}\|^2_{L^2}.
\end{aligned}
\end{equation}
Define
$$
G(t)=\int_{0}^t\frac{C}{A}\|\partial_yn_{\neq}\|^2_{L^2}\|n_{\neq}\|^{\frac{2}{7}}_{L^2}\|n_{\neq}\|_{L^4}^{\frac{12}{7}}
+\frac{C}{A}\|n_{\neq}\|^2_{L^\infty}\|n_{\neq}\|^2_{L^2}ds.
$$
By Interpolation inequality, Assumption \ref{assm:2.2} and Lemma \ref{lem:4.2}, one has
$$
\begin{aligned}
\frac{C}{A}\|\partial_yn_{\neq}\|^2_{L^2}\|n_{\neq}\|^{\frac{2}{7}}_{L^2}\|n_{\neq}\|_{L^4}^{\frac{12}{7}}&\leq \frac{C}{A}\|\partial_yn_{\neq}\|^2_{L^2}\|n_{\neq}\|^{\frac{8}{7}}_{L^2}\|n_{\neq}\|_{L^\infty}^{\frac{6}{7}}\\
&\leq \frac{CB^{\frac{6}{7}}_2\left(4C_{\neq}\|n_0\|^2_{L^2}\right)^{\frac{4}{7}}}{A}\|\nabla n_{\neq}\|^2_{L^2},
\end{aligned}
$$
and
$$
\frac{C}{A}\|n_{\neq}\|^2_{L^\infty}\|n_{\neq}\|^2_{L^2}\leq \frac{CB^2_2}{A}\|n_{\neq}\|^2_{L^2}.
$$
Then we deduce by Assumption \ref{assm:2.2} that
$$
\begin{aligned}
G(t)&\leq \frac{CB^{\frac{6}{7}}_2\left(4C_{\neq}\|n_0\|^2_{L^2}\right)^{\frac{4}{7}}}{A}\int_{0}^t\|\nabla n_{\neq}\|^2_{L^2}ds+\frac{CB^2_2}{A}\int_{0}^t\|n_{\neq}\|^2_{L^2}ds\\
&\leq CB^{\frac{6}{7}}_2C^{\frac{11}{7}}_{\neq}\|n_{0}\|^\frac{22}{7}_{L^2}+CB^2_2 C_{\neq}\|n_0\|^2_{L^2}(\epsilon_0A^{\frac{1}{2}})^{-1}.
\end{aligned}
$$
Thus, for any $A$ is large enough, one has
\begin{equation}\label{eq:4.15}
\begin{aligned}
0\leq G(t)\leq 2CB^{\frac{6}{7}}_2C^{\frac{11}{7}}_{\neq}\|n_{0}\|^\frac{22}{7}_{L^2}.
\end{aligned}
\end{equation}
By Gagliardo-Nirenberg inequality, one gets
\begin{equation}\label{eq:4.16}
-\|\partial_{yy}n^0\|^{2}_{L^2_y}\leq -\frac{\|\partial_yn^0\|^4_{L^2_y}}{C\|n^0\|^{2}_{L^2_y}}\leq -\frac{\|\partial_yn^0\|^4_{L^2_y}}{CB^2_1}.
\end{equation}
Combining to \eqref{eq:4.14}-\eqref{eq:4.16}, we have
\begin{equation}\label{eq:4.17}
\begin{aligned}
\frac{d}{dt}\left(\|\partial_yn^{0}\|^2_{L^2_y}-G(t)\right)&\leq -\frac{\|\partial_yn^0\|^4_{L^2_y}}{CAB^2_1}+\frac{C(B^2_1+M^2)}{A}\|\partial_yn^{0}\|^2_{L_y^2}\\
&\leq -\frac{\|\partial_yn^0\|^2_{L^2_y}}{CA(B^2_1+M^2)}\left(\|\partial_yn^0\|^2_{L^2_y}-G(t)-C(B^2_1+M^2)^2\right).
\end{aligned}
\end{equation}
Combining to \eqref{eq:4.15} and \eqref{eq:4.17}, to obtain
$$
\|\partial_yn^0\|^2_{L^2_y}\leq \|\partial_yn_0\|^2_{L^2_y}+C(B^2_1+M^2)^2+4CB^{\frac{6}{7}}_2C^{\frac{11}{7}}_{\neq}\|n_{0}\|^\frac{22}{7}_{L^2}\triangleq B^2_3.
$$
This completes the proof of Lemma \ref{lem:4.3}.
\end{proof}

\vskip .05in

Finally, we  prove the Proposition \ref{prop:2.3} by enhanced dissipation effect of non-shear flow.

\begin{proof}[The proof of Proposition {\rm \ref{prop:2.3}}]
We will prove (P-1) and (P-2) in the Proposition {\rm \ref{prop:2.3}} respectively.
\vskip .05in
\noindent (1) {\bf Nonzero mode $L^2\dot{H}^1$ estimate of $n_{\neq}$}. Let us multiply both sides of \eqref{eq:2.3} by $n_{\neq}$ and integrate over $\mathbb{T}^2\times \mathbb{R}$, we obtain
\begin{equation}\label{eq:4.18}
\begin{aligned}
&\frac{1}{2}\frac{d}{dt}\|n_{\neq}\|^2_{L^2}+\frac{1}{A}\|\nabla n_{\neq}\|^2_{L^2}+\int_{\mathbb{T}^2\times \mathbb{R}} y\partial_{x_1}n_{\neq}n_{\neq}dx_1dx_2dy\\
&\ \ \ +\int_{\mathbb{T}^2\times \mathbb{R}} y^2\partial_{x_2}n_{\neq}n_{\neq}dx_1dx_2dy
+\frac{1}{A}\int_{\mathbb{T}^2\times \mathbb{R}}F(c^0,c_{\neq},n^0,n_{\neq})n_{\neq}dx_1dx_2dy=0,
\end{aligned}
\end{equation}
where
\begin{equation}\label{eq:4.19}
\begin{aligned}
F(c^0,c_{\neq},n^0,n_{\neq})&=\nabla n^{0}\cdot\nabla c_{\neq}+\nabla n_{\neq}\cdot\nabla c^{0}
+\left(\nabla\cdot(n_{\neq} \nabla c_{\neq})\right)_{\neq}\\
&\ \ \ \ -n^{0}(n_{\neq}-c_{\neq})-n_{\neq}(n^0-c^0).
\end{aligned}
\end{equation}
Obviously, we deduce by integral by part that
$$
\int_{\mathbb{T}^2\times \mathbb{R}} y\partial_{x_1}n_{\neq}n_{\neq}dx_1dx_2dy=0,
$$
and
$$
\int_{\mathbb{T}^2\times \mathbb{R}} y^2\partial_{x_2}n_{\neq}n_{\neq}dx_1dx_2dy=0.
$$
Combining H\"{o}lder's inequality, Young's inequality, Gagliardo-Nirenberg inequality, Lemma \ref{lem:A.2} and \ref{lem:A.3}, one gets
$$
\begin{aligned}
\left|\frac{1}{A}\int_{\mathbb{T}^2\times \mathbb{R}}\nabla n^{0}\cdot\nabla c_{\neq}n_{\neq}dx_1dx_2dy\right|&
\lesssim \frac{1}{A}\|\partial_yn^{0}\|_{L_{y}^2}\|\partial_yc_{\neq}\|_{L^4}\|n_{\neq}\|_{L^4}\\
&\leq\frac{1}{8A}\|\nabla n_{\neq}\|^{2}_{L^2}+\frac{C}{A}\|\partial_yn^{0}\|^{\frac{8}{5}}_{L^2_y}\|n_{\neq}\|^{2}_{L^2},
\end{aligned}
$$
$$
\begin{aligned}
\left|\frac{1}{A}\int_{\mathbb{T}^2\times \mathbb{R}}\nabla n_{\neq}\cdot\nabla c^{0}n_{\neq}dx_1dx_2dy\right|&\lesssim \frac{1}{A}\|\nabla n_{\neq}\|_{L^2}\|\partial_y c^0\|_{L_y^\infty}\|n_{\neq}\|_{L^2}\\
&\leq \frac{1}{8A}\|\nabla n_{\neq}\|^2_{L^2}+\frac{C}{A}\|n^0\|^2_{L_y^2}\|n_{\neq}\|^2_{L^2},
\end{aligned}
$$
$$
\begin{aligned}
\left|\frac{1}{A}\int_{\mathbb{T}^2\times \mathbb{R}}\left(\nabla\cdot(n_{\neq} \nabla c_{\neq})\right)_{\neq}n_{\neq}dx_1dx_2dy\right|
\lesssim& \frac{1}{A} \left(\|\nabla n_{\neq} \cdot\nabla c_{\neq}\|_{L^\frac{4}{3}}+\|n_{\neq}\nabla^2c_{\neq}\|_{L^\frac{4}{3}}\right)\|n_{\neq}\|_{L^4}\\
\lesssim &\frac{1}{A}\|\nabla n_{\neq}\|_{L^2}\|n_{\neq}\|_{L^2}\|n_{\neq}\|_{L^4}+ \frac{C}{A}\|n_{\neq}\|^2_{L^4}\|n_{\neq}\|_{L^2}\\
\lesssim &\frac{1}{A}\|\nabla n_{\neq}\|^{\frac{7}{4}}_{L^2}\|n_{\neq}\|^{\frac{5}{4}}_{L^2}+\frac{C}{A}\|\nabla n_{\neq}\|^{\frac{3}{2}}_{L^2}\|n_{\neq}\|^{\frac{3}{2}}_{L^2}\\
\leq &\frac{1}{4A}\|\nabla n_{\neq}\|^2_{L^2}+\frac{C}{A}\|n_{\neq}\|^{10}_{L^2}+\frac{C}{A}\|n_{\neq}\|^{6}_{L^2},
\end{aligned}
$$
$$
\left|\frac{1}{A} \int_{\mathbb{T}^2\times \mathbb{R}}n^{0}(n_{\neq}-c_{\neq})n_{\neq}dx_1dx_2dy\right|\lesssim \frac{1}{A}\|n^{0}\|_{L_{y}^\infty}\|n_{\neq}\|^2_{L^2},
$$
and
$$
\left|\frac{1}{A} \int_{\mathbb{T}^2\times \mathbb{R}} n_{\neq}(n^0-c^0)n_{\neq}dx_1dx_2dy\right|\lesssim \frac{1}{A}\|n^{0}\|_{L_{y}^\infty}\|n_{\neq}\|^2_{L^2}.
$$
Then we deduce by \eqref{eq:4.18} that
\begin{equation}\label{eq:4.20}
\begin{aligned}
\frac{d}{dt}\|n_{\neq}\|^2_{L^2}+\frac{1}{A}\|\nabla n_{\neq}\|^2_{L^2}\lesssim & \frac{1}{A}\left(\|\partial_yn^{0}\|^{\frac{8}{5}}_{L^2_y}+\|n^0\|^2_{L_y^2}+\|n^{0}\|_{L_{y}^\infty}\right)\|n_{\neq}\|^2_{L^2}\\
&\ \ \ +\frac{1}{A}\|n_{\neq}\|^{10}_{L^2}+\frac{1}{A}\|n_{\neq}\|^{6}_{L^2}.
\end{aligned}
\end{equation}
Combining Lemma \ref{lem:4.1}-\ref{lem:4.3} and Assumption \ref{assm:2.2}, one has
$$
\frac{C}{A}\left(\|\partial_yn^{0}\|^{\frac{8}{5}}_{L^2_y}+\|n^0\|^2_{L_y^2}+\|n^{0}\|_{L_{y}^\infty}\right)\|n_{\neq}\|^2_{L^2}\\
\lesssim \frac{4C_{\neq}\left(B_3^{\frac{8}{5}}+B_1^2+B_2\right)}{A}e^{-\lambda_A t}\|n_{0}\|^2_{L^2},
$$
and
$$
\frac{C}{A}\|n_{\neq}\|^{10}_{L^2}+\frac{C}{A}\|n_{\neq}\|^{6}_{L^2}\lesssim\frac{(4C_{\neq})^5}{A}e^{-5\lambda_A t}\|n_{0}\|^{10}_{L^2}+\frac{(4C_{\neq})^3}{A}e^{-3\lambda_A t}\|n_{0}\|^{6}_{L^2}.
$$
Then by time integral in $[s,t]$ and $A$ is large enough, one has
\begin{equation}\label{eq:4.21}
\begin{aligned}
&\int_{s}^t\frac{C}{A}\left(\|\partial_yn^{0}\|^{\frac{8}{5}}_{L^2_y}+\|n^0\|^2_{L_y^2}+\|n^{0}\|_{L_{y}^\infty}\right)
\|n_{\neq}\|^2_{L^2}d\tau\\
\lesssim& \frac{4C_{\neq}\left(B_3^{\frac{8}{5}}+B_1^2+B_2\right)}{\epsilon_0A^{\frac{1}{2}}}e^{-\lambda_A s}\|n_{0}\|^2_{L^2}\\
\leq &2C_{\neq}e^{-\lambda_A s}\|n_{0}\|^2_{L^2},
\end{aligned}
\end{equation}
and
\begin{equation}\label{eq:4.22}
\begin{aligned}
&\int_{s}^t\frac{C}{A}\|n_{\neq}\|^{10}_{L^2}+\frac{C}{A}\|n_{\neq}\|^{6}_{L^2}d\tau\\
\lesssim& \left(\frac{(4C_{\neq})^5\|n_{0}\|^{8}_{L^2}}{5\varepsilon_0 A^{\frac{1}{2}}}+\frac{(4C_{\neq})^3\|n_{0}\|^{4}_{L^2}}{3\varepsilon_0 A^{\frac{1}{2}}} \right)e^{-\lambda_A s}\|n_{0}\|^{2}_{L^2}\\
\leq &2C_{\neq}e^{-\lambda_A s}\|n_{0}\|^2_{L^2}.
\end{aligned}
\end{equation}
Combining \eqref{eq:4.20}-\eqref{eq:4.22} and Assumption \ref{assm:2.2}, one has
$$
\begin{aligned}
&\|n_{\neq}(t)\|^2_{L^2}+\frac{1}{A}\int_{s}^t\|\nabla n_{\neq}\|^2_{L^2}d\tau\\
\leq & \int_{s}^t\frac{C}{A}\left(\|\partial_yn^{0}\|^{\frac{8}{5}}_{L^2_y}+\|n^0\|^2_{L_y^2}+\|n^{0}\|_{L_{y}^\infty}\right)
\|n_{\neq}\|^2_{L^2}d\tau\\
&+\int_{s}^t\frac{C}{A}\|n_{\neq}\|^{10}_{L^2}+\frac{C}{A}\|n_{\neq}\|^{6}_{L^2}d\tau+\|n_{\neq}(s)\|^2_{L^2}\\
\leq &8C_{\neq}e^{-\lambda_A s}\|n_{0}\|^2_{L^2},
\end{aligned}
$$
Then we have
$$
\frac{1}{A}\int_{s}^t\|\nabla n_{\neq}\|^2_{L^2}d\tau\leq 8C_{\neq}e^{-\lambda_A s}\|n_{0}\|^2_{L^2}.
$$
\vskip .05in
\noindent (2) {\bf Nonzero mode $L^\infty L^2$ estimate of $n_{\neq}$}. Denote
$$
\mathcal{S}_{t}=e^{-tL_A},
$$
and combining Duhamel's principle, \eqref{eq:2.3} and \eqref{eq:4.19}, one has
$$
n_{\neq}(s+t)=\mathcal{S}_{t}n_{\neq}(s)
-\frac{1}{A}\int_{s}^{s+t}\mathcal{S}_{t+s-\tau}F(c^0,c_{\neq},n^0,n_{\neq})d\tau.
$$
Then we have
\begin{equation}\label{eq:4.23}
\begin{aligned}
\|n_{\neq}(t+s)\|_{L^2}&\leq \|\mathcal{S}_{t}n_{\neq}(s)\|_{L^2}+\frac{C}{A}\int_{s}^{s+t}\|\partial_yn^{0}\partial_yc_{\neq}\|_{L^2}+\|\partial_yn_{\neq}\partial_y c^{0}\|_{L^2}\\
&+\|\left(\nabla\cdot(n_{\neq} \nabla c_{\neq})\right)_{\neq}\|_{L^2}+\|n^{0}(n_{\neq}-c_{\neq})\|_{L^2}+\|n_{\neq}(n^0-c^0)\|_{L^2}d\tau.
\end{aligned}
\end{equation}
By the Theorem \ref{thm:1.0}, one gets
\begin{equation}\label{eq:4.24}
\|\mathcal{S}_{t}n_{\neq}(s+t)\|_{L^2}\lesssim e^{-\lambda_A t}\|n_{\neq}(s)\|_{L^2}.
\end{equation}
Combining H\"{o}lder's inequality, Gagliardo-Nirenberg inequality and Lemma \ref{lem:A.3}, one has
$$
\|\partial_yn^{0}\partial_yc_{\neq}\|_{L^2}\lesssim \|\partial_yn^{0}\|_{L^2_y}\|n_{\neq}\|_{L^2}+\|\partial_yn^{0}\|_{L^2_y}\|n_{\neq}\|^{\frac{1}{4}}_{L^2}\|\nabla n_{\neq}\|_{L^2}^{\frac{3}{4}}.
$$
By H\"{o}lder's inequality, Lemma \ref{lem:4.3} and Assumption \ref{assm:2.2}, one gets
$$
\frac{C}{A}\int_{s}^{s+t}\|\partial_yn^{0}\|_{L^2_y}\|n_{\neq}\|_{L^2}d\tau\lesssim \frac{B_3C^{\frac{1}{2}}_{\neq}t}{A}e^{-\frac{\lambda_A}{2} s}\|n_0\|_{L^2},
$$
and
$$
\begin{aligned}
&\frac{C}{A}\int_{s}^{s+t}\|\partial_yn^{0}\|_{L^2_y}\|n_{\neq}\|^{\frac{1}{4}}_{L^2}\|\nabla n_{\neq}\|_{L^2}^{\frac{3}{4}}d\tau\\
\lesssim& \left( \frac{1}{A}\int_{s}^{s+t}\|\partial_yn^{0}\|^{\frac{8}{5}}_{L^2_y}d\tau \right)^{\frac{5}{8}}\left( \frac{1}{A}\int_{s}^{s+t}\|n_{\neq}\|^{\frac{2}{3}}_{L^2}\|\nabla n_{\neq}\|_{L^2}^{2}d\tau \right)^{\frac{3}{8}}\\
\lesssim& B_3\left(\frac{t}{A}\right)^{\frac{5}{8}}\left(4C_{\neq}e^{-\lambda_As}\|n_0\|^2_{L^2}\right)^{\frac{1}{8}}\left( \frac{1}{A}\int_{s}^{s+t}\|\nabla n_{\neq}\|_{L^2}^{2}d\tau \right)^{\frac{3}{8}}\\
\lesssim& B_3C^{\frac{1}{2}}_{\neq}\left(\frac{t}{A}\right)^{\frac{5}{8}}e^{-\frac{\lambda_A}{2}s}\|n_0\|_{L^2}.
\end{aligned}
$$
Then we have
\begin{equation}\label{eq:4.25}
\frac{C}{A}\int_{s}^{s+t}\|\partial_yn^{0}\partial_yc_{\neq}\|_{L^2}d\tau\lesssim \left(\frac{B_3C^{\frac{1}{2}}_{\neq}t}{A}+B_3C^{\frac{1}{2}}_{\neq}
\left(\frac{t}{A}\right)^{\frac{5}{8}}\right)e^{-\frac{\lambda_A}{2}s}\|n_0\|_{L^2}.
\end{equation}
Combining H\"{o}lder's inequality, Lemma \ref{lem:4.1}, Lemma \ref{lem:A.2} and Assumption \ref{assm:2.2}, one has
\begin{equation}\label{eq:4.26}
\begin{aligned}
\frac{C}{A}\int_{s}^{s+t}\|\partial_yn_{\neq}\partial_y c^{0}\|_{L^2}d\tau&\lesssim\frac{1}{A}\int_{s}^{s+t}\|\nabla n_{\neq}\|_{L^2}\|n^{0}\|_{L_y^2}d\tau\\
&\lesssim \left(\frac{1}{A}\int_{s}^{s+t}\|n^{0}\|^2_{L_y^2}d\tau\right)^{\frac{1}{2}}\left(\frac{1}{A}\int_{s}^{s+t}\|\nabla n_{\neq}\|^2_{L^2}d\tau\right)^{\frac{1}{2}}\\
&\lesssim B_1C^{\frac{1}{2}}_{\neq}\left(\frac{t}{A}\right)^{\frac{1}{2}}e^{-\frac{\lambda_A}{2}s}\|n_0\|_{L^2}.
\end{aligned}
\end{equation}
By H\"{o}lder's inequality, Gagliardo-Nirenberg inequality and Lemma \ref{lem:A.3}, one has
$$
\|\left(\nabla\cdot(n_{\neq} \nabla c_{\neq})\right)_{\neq}\|_{L^2}
\lesssim\|\nabla n_{\neq}\|^{\frac{7}{4}}_{L^2}\|n_{\neq}\|^{\frac{1}{4}}_{L^2}+\|\nabla n_{\neq}\|_{L^2}\|n_{\neq}\|_{L^2}+\|n_{\neq}\|_{L^\infty}\|n_{\neq}\|_{L^2},
$$
then we deduce by H\"{o}lder's inequality, Lemma \ref{lem:4.2} and Assumption \ref{assm:2.2} that
\begin{equation}\label{eq:4.27}
\begin{aligned}
&\frac{C}{A}\int_{s}^{s+t}\|\left(\nabla\cdot(n_{\neq} \nabla c_{\neq})\right)_{\neq}\|_{L^2}d\tau\\
\lesssim& \left(C_{\neq}\|n_0\|_{L^2}\left(\frac{t}{A}\right)^{\frac{1}{8}}
+C_{\neq}\|n_0\|_{L^2}\left(\frac{t}{A}\right)^{\frac{1}{2}}
+\epsilon_0^{-1}B_2C^{\frac{1}{2}}_{\neq}\left(\frac{t}{A}\right)^{\frac{1}{2}}\right)
e^{-\frac{\lambda_A}{2}s}\|n_0\|_{L^2}.
\end{aligned}
\end{equation}
By similar to argument, one has
\begin{equation}\label{eq:4.28}
\frac{C}{A}\int_{s}^{s+t}\|n^{0}(n_{\neq}-c_{\neq})\|_{L^2}d\tau
\lesssim\frac{B_2(C_{\neq})^{\frac{1}{2}}}{\epsilon_0A^{\frac{1}{2}}}e^{-\frac{\lambda_A}{2}s}\|n_0\|_{L^2},
\end{equation}
and
\begin{equation}\label{eq:4.29}
\frac{C}{A}\int_{s}^{s+t}\|n_{\neq}(n^0-c^0)\|_{L^2}d\tau
\lesssim B^{\frac{1}{4}}_3B^{\frac{3}{4}}_1C^{\frac{1}{2}}_{\neq}
\left(\frac{t}{A}\right)^{\frac{5}{8}}e^{-\frac{\lambda_A}{2}s}\|n_0\|_{L^2}.
\end{equation}
Combining \eqref{eq:4.23}-\eqref{eq:4.29}, we have
\begin{equation}\label{eq:4.30}
\begin{aligned}
&\|n_{\neq}(t+s)\|_{L^2}\\
 \lesssim& e^{-\lambda_A t+\pi/2}\|n_{\neq}(s)\|_{L^2}
+\left(\frac{B_3C^{\frac{1}{2}}_{\neq}t}{A}+B_3C^{\frac{1}{2}}_{\neq}\left(\frac{t}{A}\right)^{\frac{5}{8}}
+B_1C^{\frac{1}{2}}_{\neq}\left(\frac{t}{A}\right)^{\frac{1}{2}}\right)e^{-\frac{\lambda_A}{2}s}\|n_0\|_{L^2}\\
&+ \left(C_{\neq}\|n_0\|_{L^2}\left(\frac{t}{A}\right)^{\frac{1}{8}}
+C_{\neq}\|n_0\|_{L^2}\left(\frac{t}{A}\right)^{\frac{1}{2}}
+\epsilon_0^{-1}B_2C^{\frac{1}{2}}_{\neq}\left(\frac{t}{A}\right)^{\frac{1}{2}}\right)
e^{-\frac{\lambda_A}{2}s}\|n_0\|_{L^2}\\
&+\left(\frac{B_2(C_{\neq})^{\frac{1}{2}}}{\epsilon_0A^{\frac{1}{2}}}+B^{\frac{1}{4}}_3B^{\frac{3}{4}}_1C^{\frac{1}{2}}_{\neq}
\left(\frac{t}{A}\right)^{\frac{5}{8}}\right)e^{-\frac{\lambda_A}{2}s}\|n_0\|_{L^2}.
\end{aligned}
\end{equation}
%We choose $N_0\in \mathbb{Z}^{+}$ is large and define
Taking $\tau^\ast=8\lambda_A^{-1}$, and combining \eqref{eq:4.30} and $A$ is large enough, we imply
\begin{equation}\label{eq:4.31}
\|n_{\neq}(\tau^\ast+s)\|_{L^2}\leq \frac{1}{2}e^{-4}\|n_{\neq}(s)\|_{L^2}+\frac{1}{2}e^{-4}e^{-\frac{\lambda_A}{2}s}\|n_0\|_{L^2},
\end{equation}
and if $s=0$, one gets
$$
\|n_{\neq}(\tau^\ast)\|_{L^2}\leq  e^{-4}\|n_0\|_{L^2}.
$$
Then we deduce by mathematical induction and \eqref{eq:4.31} that for any $m\in \mathbb{Z}^+$, one has
%and by mathematical induction and \eqref{eq:4.31}, we can imply that for any $m\in \mathbb{Z}^+$, one has
\begin{equation}\label{eq:4.32}
\|n_{\neq}(m\tau^\ast)\|_{L^2}\leq e^{-4m}\|n_0\|_{L^2}.
\end{equation}
By the same argument, we know that for any $t\in [0,T^\ast]$, there exist $\tau_0$ and $m\in \mathbb{Z}^+$, such that
$$
t=m\tau^\ast+\tau_0, \ \ \ 0\leq \tau_0\leq\tau^\ast,
$$
combining local estimate of solution to Equation \eqref{eq:2.3} and \eqref{eq:4.32}, one gets
$$
\|n_{\neq}(t)\|_{L^2}\leq 2\|n_{\neq}(m\tau^\ast)\|_{L^2}\leq 2e^{-4m}\|n_0\|_{L^2}.
$$
Then we obtain
$$
\|n_{\neq}(t)\|^2_{L^2}\leq 4e^{-8m}\|n_0\|_{L^2}\leq 4e^{-8\frac{t-\tau_0}{\tau^\ast}}\|n_0\|_{L^2}\leq 4e^{8}e^{-\lambda_A t}\|n_0\|^2_{L^2}\leq 2C_{\neq}e^{-\lambda_A t}\|n_0\|^2_{L^2},
$$
where $C_{\neq}\geq 2e^{8}$. This completes the proof of Proposition \ref{prop:2.3}.
\end{proof}

\vskip .2in
\appendix

\section{}

In this section, we introduce some useful mathematical tools and necessary supplements, which are convenient for us to read this paper.

\subsection{Moser-Alikakos iteration}
The Moser-Alikakos type iteration\cite{Alikakos.1979} is useful to establish the $L^\infty$  estimate of solution, where we give it in our convenient form.

\begin{lemma}[Moser-Alikakos iteration]\label{lem:A.1}
For function sequence $\{y_k(t)\}_{k=0}^\infty$, $y_k(t)\geq 0$, if $y_k(t)$ satisfies
\begin{equation}\label{eq:A.1}
 y_k(t)\leq a_k\max\left\{\sup_{t\geq 0}y^2_{k-1},y_k(0),1\right\},
\end{equation}
\begin{equation}\label{eq:A.2}
 y_0(t)\leq M_1,\ \ \ \  y_k(0)\leq M_2,\ \ \ a_k=C\eta^k,
\end{equation}
where $M_1, M_2, C$ and $\eta$ are positive constant. Then there exists a positive constant
$$
K_\infty=2C^{\frac{1}{2}}\eta\max\{ M_1^{\frac{1}{2}},M_2^{\frac{1}{4}},1\},
$$
such that for $k$ is large enough, one has
$$
\sup_{t\geq 0}y^{\frac{1}{2^{k+1}}}_k(t)\leq K_\infty.
$$
\end{lemma}

\begin{proof}
Combining \eqref{eq:A.1} and \eqref{eq:A.2}, we deduce that
\begin{equation}\label{eq:A.3}
\begin{aligned}
y_{k}(t)&\leq a_k\max\{ \sup_{t\geq 0}y^{2}_{k-1},y_{k}(0),1\}\\
&\leq a_k\max\{ \sup_{t\geq 0}(a_{k-1}\max\{ \sup_{t\geq 0}y^{2}_{k-2},y_{k-1}(0),1\})^2,M_2,1\}\\
&\leq a_k(a_{k-1}\max\{ \sup_{t\geq 0}y^{2}_{k-2},M_2,1\})^2\\
&\leq a_ka^2_{k-1}\max\{ \sup_{t\geq 0}y^{4}_{k-2},K^2\},
\end{aligned}
\end{equation}
where
$$
K=\max\{M_2,1\}.
$$
Repeat the process of \eqref{eq:A.3}, one gets
\begin{equation}\label{eq:A.4}
y_{k}(t)\leq a_k(a_{k-1})^2(a_{k-2})^{2^2}\cdots (a_{1})^{2^{k-1}}\max\{ \sup_{t\geq 0}y^{2^k}_{0},K^{2^{k-1}}\}
\end{equation}
Since $a_k=C\eta^k$, then
$$
a_k(a_{k-1})^2(a_{k-2})^{2^2}\cdots (a_{1})^{2^{k-1}}=C^{2^k-1}\eta^{2^{k+1}-2-k}.
$$
Then we can imply by \eqref{eq:A.2} and \eqref{eq:A.4} that
$$
\begin{aligned}
y_{k}(t)&\leq C^{2^k-1}\eta^{2^{k+1}-2-k}\max\{ \sup_{t\geq 0}y^{2^k}_{0},K^{2^{k-1}}\}\\
&\leq C^{2^k-1}\eta^{2^{k+1}-2-k}\max\{ M_1^{2^k},K^{2^{k-1}}\}.
\end{aligned}
$$
Therefore, if $k$ is large enough, one has
$$
\begin{aligned}
\sup_{t\geq 0}y^{\frac{1}{2^{k+1}}}_k(t)&\leq C^{\frac{1}{2}-\frac{1}{2^{k+1}}}\eta^{1-\frac{2+k}{2^{k+1}}}\max\{ M_1^{\frac{1}{2}},K^{\frac{1}{4}}\}\\
&\leq 2C^{\frac{1}{2}}\eta\max\{ M_1^{\frac{1}{2}},M_2^{\frac{1}{4}},1\}\triangleq K_\infty.
\end{aligned}
$$
This completes the proof of Lemma \ref{lem:A.1}.
\end{proof}

\subsection{Elliptic estimate}
In this paper, we need some useful elliptic estimates. Considering
elliptic equation
\begin{equation}\label{eq:A.5}
-\Delta c=n-c, \ \ \ x\in \Omega.
\end{equation}
Here we give the following lemmas.

%In this section, we establish some useful elliptic estimate, we consider the following elliptic equation
%\begin{equation}\label{eq:A.5}
%-\Delta c=n-c, \ \ \ x\in \Omega.
%\end{equation}

\begin{lemma}\label{lem:A.2}
If $\Omega=\mathbb{R}$, the solution of Equation \eqref{eq:A.5} hold that
$$
\|c\|_{L^p}\lesssim\|n\|_{L^p}, \ \ \ p\geq2,
$$
and
$$
\|\partial_xc\|_{L^2}+\|\partial_xc\|_{L^4}+\|\partial_xc\|_{L^\infty}+\|\partial^2_xc\|_{L^2}\lesssim\|n\|_{L^2}.
$$
\end{lemma}

\vskip .05in

\begin{proof}
Let us multiply both sides of \eqref{eq:A.5} by $\partial^2_{x} c$ and integrate over $\mathbb{R}$, to obtain
\begin{equation}\label{eq:A.6}
-\|\partial^2_xc\|^2_{L^2}-\|\partial_xc\|^2_{L^2}=\int_{\mathbb{R}}n\partial^2_xcdx,
\end{equation}
and by H\"{o}lder's inequality and Young's inequality, one has
$$
\left|\int_{\mathbb{R}}n\partial^2_xcdx\right|\lesssim\|n\|_{L^2}\|\partial^2_xc\|_{L^2}\leq C\|n\|^2_{L^2}+\frac{1}{2}\|\partial^2_xc\|^2_{L^2}.
$$
Then we deduce by \eqref{eq:A.6} that
\begin{equation}\label{eq:A.7}
\|\partial^2_xc\|^2_{L^2}+\|\partial_xc\|^2_{L^2}\lesssim\|n\|^2_{L^2}.
\end{equation}
By using the Gagliardo-Nirenberg inequality and \eqref{eq:A.7}, we have
\begin{equation}\label{eq:A.8}
\|\partial_xc\|_{L^4}\lesssim\|\partial_xc\|^{\frac{3}{4}}_{L^2}\|\partial^2_xc\|^{\frac{1}{4}}_{L^2}\lesssim\|n\|_{L^2},
\end{equation}
and
\begin{equation}\label{eq:A.9}
\|\partial_xc\|_{L^\infty}\lesssim\|\partial_xc\|^{\frac{1}{2}}_{L^2}\|\partial^2_xc\|^{\frac{1}{2}}_{L^2}\lesssim\|n\|_{L^2},
\end{equation}
Let us multiply both sides of \eqref{eq:A.5}) by $|c|^{p-2}c$ and integrate over $\mathbb{R}$, to obtain
$$
-\int_{\mathbb{R}}\Delta c|c|^{p-2}cdx+\int_{\mathbb{R}}|c|^{p}dx=\int_{\mathbb{R}}n|c|^{p-2}cdx,
$$
then by H\"{o}lder's inequality, one has
\begin{equation}\label{eq:A.10}
\frac{2}{p}\int_{\mathbb{R}} (\partial_x |c|^{\frac{p}{2}})^2dx+\|c\|^p_{L^p}=\int_{\mathbb{R}}n|c|^{p-2}cdx\lesssim\|n\|_{L^p}\|c\|^{p-1}_{L^p}.
\end{equation}
Combining \eqref{eq:A.7}-\eqref{eq:A.10}, we finish the proof of Lemma \ref{lem:A.2}.
\end{proof}

\vskip .03in

\begin{lemma}\label{lem:A.3}
If $\Omega=\mathbb{T}^2\times \mathbb{R}$, the solution of Equation \eqref{eq:A.5} hold that
$$
\|\nabla^2c\|_{L^2}+\|\nabla c\|_{L^2}+\|\nabla c\|_{L^4}\lesssim\|n\|_{L^2},
$$
and for any $p\geq2,k\geq2, k\in \mathbb{Z}$, one has
$$
\|c\|_{L^p}\lesssim\|n\|_{L^p}, \ \ \ \   \|\nabla^kc\|_{L^p} \lesssim\|\nabla^{k-2}n\|_{L^p}+\|n\|_{L^p}.
$$
\end{lemma}

\vskip .05in

\begin{proof}
Let us multiply both sides of \eqref{eq:A.5} by $\Delta c$ and integrate over $\mathbb{T}^2\times \mathbb{R}$, to obtain
\begin{equation}\label{eq:A.11}
-\|\nabla^2 c\|^2_{L^2}-\|\nabla c\|^2_{L^2}=\int_{\mathbb{T}^2\times\mathbb{R}}n \Delta cdx,
\end{equation}
and by H\"{o}lder's inequality and Young's inequality, one has
$$
\left|\int_{\mathbb{T}^2\times\mathbb{R}}n \Delta cdx\right|\lesssim\|n\|_{L^2}\|\nabla^2c\|_{L^2}\leq C\|n\|^2_{L^2}+\frac{1}{2}\|\nabla^2c\|^2_{L^2}.
$$
Then we deduce by \eqref{eq:A.11} that
\begin{equation}\label{eq:A.12}
\|\nabla^2c\|^2_{L^2}+\|\nabla c\|^2_{L^2}\lesssim\|n\|^2_{L^2}.
\end{equation}
Using the Gagliardo-Nirenberg inequality and \eqref{eq:A.12}, we have
\begin{equation}\label{eq:A.13}
\|\nabla c\|_{L^4}\lesssim\|\nabla c\|^{\frac{1}{4}}_{L^2}\|\nabla^2c\|^{\frac{3}{4}}_{L^2}\lesssim\|n\|_{L^2},
\end{equation}
and by the similar argument with Lemma \ref{lem:A.2}, one has
\begin{equation}\label{eq:A.14}
\|c\|_{L^p}\lesssim\|n\|_{L^p}.
\end{equation}
The operator $\nabla^{k-2}$ is applied to Equation \eqref{eq:A.5}, one has
\begin{equation}\label{eq:A.15}
-\nabla^{k-2}\Delta c=\nabla^{k-2}n-\nabla^{k-2} c.
\end{equation}
Let us multiply both sides of \eqref{eq:A.15} by $|\nabla^k c|^{p-2}\nabla^k c$ and integrate over $\mathbb{T}^2\times[0,1]$, to obtain
\begin{equation}\label{eq:A.16}
\begin{aligned}
-\int_{\mathbb{T}^2\times\mathbb{R}}\nabla^{k-2}\Delta c|\nabla^k c|^{p-2}\nabla^k cdx&=\int_{\mathbb{T}^2\times\mathbb{R}}\nabla^{k-2}n|\nabla^k c|^{p-2}\nabla^k cdx\\
&-\int_{\mathbb{T}^2\times\mathbb{R}}\nabla^{k-2}c|\nabla^k c|^{p-2}\nabla^k cdx.
\end{aligned}
\end{equation}
By integrate by part, one has
$$
\int_{\mathbb{T}^2\times\mathbb{R}}\nabla^{k-2}\Delta c|\nabla^k c|^{p-2}\nabla^k cdx=\|\nabla^k c\|^p_{L^p},
$$
and by H\"{o}lder's inequality, one has
$$
\left|\int_{\mathbb{T}^2\times\mathbb{R}}\nabla^{k-2}n|\nabla^k c|^{p-2}\nabla^k cdx\right|\lesssim\|\nabla^{k-2}n\|_{L^p}\|\nabla^{k}c\|^{p-1}_{L^p},
$$
and
$$
\left|\int_{\mathbb{T}^2\times\mathbb{R}}\nabla^{k-2}c|\nabla^k c|^{p-2}\nabla^k cdx\right|\lesssim\|\nabla^{k-2}c\|_{L^p}\|\nabla^{k}c\|^{p-1}_{L^p}.
$$
Therefore, we imply by \eqref{eq:A.16} that
\begin{equation}\label{eq:A.17}
\|\nabla^{k}c\|_{L^p}\lesssim\|\nabla^{k-2}n\|_{L^p}+\|\nabla^{k-2}c\|_{L^p}.
\end{equation}
We deduce by \eqref{eq:A.14}, \eqref{eq:A.17} and Gagliardo-Nirenberg inequality that
$$
\|\nabla^{2}c\|_{L^p}\lesssim\| n\|_{L^p}+\|c\|_{L^p}\lesssim\| n\|_{L^p},
$$
and
$$
\|\nabla c\|_{L^p}\lesssim\|c\|^{\frac{1}{2}}_{L^p}\|\nabla^{2}c\|^{\frac{1}{2}}_{L^p}\lesssim\| n\|_{L^p}.
$$
Then we deduce by \eqref{eq:A.17} that
\begin{equation}\label{eq:A.18}
\|\nabla^{k}c\|_{L^p}\lesssim\|\nabla^{k-2}n\|_{L^p}+\|n\|_{L^p}.
\end{equation}
Combining \eqref{eq:A.11}-\eqref{eq:A.14} and \eqref{eq:A.18}, we finish the proof Lemma \ref{lem:A.3}.
\end{proof}

\vskip .05in

\subsection{Local estimate} The local estimate  makes the Assumption \ref{ass:3.1} reasonable, which is the premise of bootstrap argument. Here we give the local estimate of solution to Equation \eqref{eq:2.3}.

\begin{lemma}[Local estimate of $n_{\neq}$]\label{lem:A.4}
Let $n_{\neq}$ be a solution of Equation \eqref{eq:2.3} with initial data $n_0\geq 0$ and $n_0\in H^s(\mathbb{T}^2\times \mathbb{R}), s>\frac{7}{2}$. Then there exist time $t^\ast=8\lambda^{-1}_A$, $A$ is large enough and a  positive constant $C_{\neq}$, such that for any $0\leq s\leq t\leq t^\ast$, one has
$$
\frac{1}{A}\int_{s}^{t}\|\nabla n_{\neq}\|^2_{L^2}d\tau\leq 16C_{\neq}e^{-\lambda_A s}\|n_{0}\|^2_{L^2},
$$
and for any $0\leq t\leq t^\ast$, one has
$$
\|n_{\neq}(t)\|^2_{L^2}\leq 4C_{\neq}e^{-\lambda_A t}\|n_{0}\|^2_{L^2},
$$
where $\lambda_A$ is defined in \eqref{eq:1.6}.% and $N_0$ is fixed in.
\end{lemma}

%\vskip .05in

\begin{proof}
We establish the $L^2$ estimate of $n$, let us multiply both sides of \eqref{eq:2.1} by $n$ and integrate over $\mathbb{T}^2\times \mathbb{R}$, to obtain
\begin{equation}\label{eq:A.19}
\frac{1}{2}\frac{d}{dt}\|n\|^2_{L^2}+\frac{1}{A}\|\nabla n\|^2_{L^2}+\frac{1}{A}\int_{\mathbb{T}^2\times \mathbb{R}}\nabla(n\nabla c)ndx_1dx_2dy=0.
\end{equation}
By H\"{o}lder's inequality, Gagliardo-Nirenberg inequality, Young's inequality and Lemma \ref{lem:A.3}, the third term of the left-hand side of \eqref{eq:A.19} can be estimated as
$$
\left|\frac{1}{A}\int_{\mathbb{T}^2\times \mathbb{R}}\nabla(n\nabla c)ndx_1dx_2dy\right|
\lesssim \frac{1}{A}\|\Delta c\|_{L^2}\|n\|^2_{L^4}
\leq \frac{1}{2A}\|\nabla n\|^{2}_{L^2}+\frac{C}{A}\|n\|^{6}_{L^2}.
$$
Combining \eqref{eq:A.19}, we have
\begin{equation}\label{eq:A.20}
\frac{d}{dt}\|n\|^2_{L^2}+\frac{1}{A}\|\nabla n\|^2_{L^2}\leq \frac{C}{A}\|n\|^{6}_{L^2}.
\end{equation}
Denote $B_0=\|n_0\|^2_{L^2}$, we deduce by solving the differential inequality in \eqref{eq:A.20} that
$$
\|n(t,\cdot)\|^2_{L^2}\leq B_0\left(1-2CA^{-1}B_0^2t\right)^{-\frac{1}{2}}
$$
Then for any $0\leq t\leq \left(3CB^2_0\right)^{-1}A$, one has
$$
\|n\|^2_{L^2}\leq 4\|n_0\|^2_{L^2}.
$$
Choose $A$ is large, such that $8\lambda^{-1}_A\leq\left(3CB^2_0\right)^{-1}A$, then for any $0\leq t\leq 8\lambda^{-1}_A$, one has
\begin{equation}\label{eq:A.21}
\|n^{0}\|^2_{L^2_y}+\|n_{\neq}(t)\|^2_{L^2}=\|n(t)\|^2_{L^2}\leq 4\|n_0\|^2_{L^2}\leq 4C_{\neq}e^{-\lambda_A t}\|n_0\|^2_{L^2},
\end{equation}
where $C_{\neq}\geq e^{8}$.
Combining \eqref{eq:A.20} and \eqref{eq:A.21}, we imply that for any $0\leq s\leq t\leq t^\ast$, one has
\begin{equation}\label{eq:A.22}
\begin{aligned}
\frac{1}{A}\int_{s}^{t}\|\nabla n\|^2_{L^2}d\tau &\leq \frac{C}{A}\int_{s}^{t}\|n\|^6_{L^2}d\tau+\|n(s)\|^2_{L^2}\\
&\leq \frac{C(4C_{\neq})^3\|n_0\|^4_{L^2}}{A}\int_{s}^{t}e^{-3\lambda_A \tau}d\tau\|n_0\|^2_{L^2}+4C_{\neq}e^{-\lambda_A s}\|n_0\|^2_{L^2}\\
&\leq \frac{C(4C_{\neq})^3\|n_0\|^4_{L^2}}{3\lambda_AA}e^{-3\lambda_A s}\|n_0\|^2_{L^2}+2C_{\neq}e^{-\lambda_A s}\|n_0\|^2_{L^2}.
\end{aligned}
\end{equation}
And base on \eqref{eq:A.21} and the energy estimate of Equation \eqref{eq:2.2}, see \eqref{eq:4.14}, one gets
\begin{equation}\label{eq:A.23}
\begin{aligned}
\frac{1}{A}\int_{s}^{t}\|\partial_y n^0\|^2_{L^2_y}d\tau&\leq \frac{C}{A}\int_{s}^{t}\|n^{0}\|^4_{L^2_y}d\tau+\frac{1}{32C_{\neq}A}\int_{s}^{t}\|\nabla n_{\neq}\|^2_{L^2}d\tau\\
&+\frac{C}{A}\int_{s}^{t}\|n_{\neq}\|^{10}_{L^2}d\tau+\|n^{0}(s)\|^2_{L^2_y}\\
&\leq \frac{C(4C_{\neq})^2\|n_0\|^2_{L^2}}{2\lambda_AA}e^{-2\lambda_A s}\|n_0\|^2_{L^2}+\frac{1}{32C_{\neq}A}\int_{s}^{t}\|\nabla n_{\neq}\|^2_{L^2}d\tau\\
&+\frac{C(4C_{\neq})^5\|n_0\|^8_{L^2}}{5\lambda_AA}e^{-5\lambda_A s}\|n_0\|^2_{L^2}+2C_{\neq}e^{-\lambda_A s}\|n_0\|^2_{L^2}
\end{aligned}
\end{equation}
Combining \eqref{eq:A.22} and \eqref{eq:A.23}, we obtain
$$
\begin{aligned}
\frac{1}{A}\int_{s}^{t}\|\nabla n_{\neq}\|^2_{L^2}d\tau&\leq \frac{2}{A}\int_{s}^{t}\|\nabla n\|^2_{L^2}d\tau+\frac{2}{A}\int_{s}^{t}\|\partial_y n^0\|^2_{L^2_y}d\tau\\
&\leq \frac{C(4C_{\neq})^3\|n_0\|^4_{L^2}}{3\lambda_AA}e^{-\lambda_A s}\|n_0\|^2_{L^2}+\frac{C(4C_{\neq})^5\|n_0\|^8_{L^2}}{5\lambda_AA}e^{-\lambda_A s}\|n_0\|^2_{L^2}\\
&+ \frac{C(4C_{\neq})^2\|n_0\|^2_{L^2}}{2\lambda_AA}e^{-2\lambda_A s}\|n_0\|^2_{L^2}+\frac{1}{16C_{\neq}A}\int_{s}^{t}\|\nabla n_{\neq}\|^2_{L^2}d\tau\\
&+8C_{\neq}e^{-\lambda_A s}\|n_0\|^2_{L^2}.
\end{aligned}
$$
Therefore, when $A$ is large enough, one has
\begin{equation}\label{eq:A.24}
\frac{1}{A}\int_{s}^{t}\|\nabla n_{\neq}\|^2_{L^2}d\tau \leq 16C_{\neq}e^{-\lambda_A s}\|n_0\|^2_{L^2}.
\end{equation}
Combining \eqref{eq:A.21} and \eqref{eq:A.24}, we finish the proof of Lemma \ref{lem:A.4}.
\end{proof}

\subsection{Blow up in finite time}
Here the goal is to construction of examples where solutions to the Keller-Segel equation on $\mathbb{T}^2\times \mathbb{R}$ without advection blow up in finite time.

%\vskip .1in

\begin{lemma}\label{lem:A.5}
There exist $n_0\geq 0, n_0\in H^s(\mathbb{T}^2\times \mathbb{R}),s>\frac{7}{2}$, such that the solution of Equation \eqref{eq:1.1} with $A=0$ on $\mathbb{T}^2\times \mathbb{R}$ blows up in finite time.
\end{lemma}

%\vskip .1in
\begin{proof}
Base on the second equation in Equation \eqref{eq:1.1} and singular integral, one has
\begin{equation}\label{eq:A.25}
\nabla c=\nabla(I-\Delta)^{-1}n=\sum_{k\in \mathbb{Z}^3}\int_{\mathbb{T}^2\times \mathbb{R}}\nabla K(x-y+k)n(t,y)dy,
\end{equation}
where $n(t,y)$ is extended periodically to all $\mathbb{R}^3$, $k=(k_1,k_2,0)$ and
$$
K(x)=(4\pi)^{-\frac{3}{2}}\int_{0}^{\infty}e^{-s-\frac{|x|^2}{4s}}s^{-\frac{3}{2}}ds.
$$
Then we deduce by calculation that the \eqref{eq:A.25} can be written as
\begin{equation}\label{eq:A.26}
\nabla c=-\sum_{k\in \mathbb{Z}^3}\int_{\mathbb{T}^2\times \mathbb{R}}(x-y+k)\widetilde{E}(x-y+k)n(t,y)dy,
\end{equation}
where
\begin{equation}\label{eq:A.27}
\widetilde{E}(x)=\frac{C_{d}}{|x|^3}\int_{0}^{\infty}e^{-s-\frac{|x|^2}{4s}}s^{\frac{1}{2}}ds,\ \ \ C_{d}=\frac{\pi^{-\frac{3}{2}}}{2}.
\end{equation}
Define the quantity
$$
I(t)=\int_{\mathbb{T}^2\times \mathbb{R}}\frac{|x|^2}{2}n(t,x)dx,
$$
then we deduce by \eqref{eq:1.1} and \eqref{eq:A.26} that
\begin{equation}\label{eq:A.28}
\begin{aligned}
\frac{d}{dt}I(t)&=\int_{\mathbb{T}^2\times \mathbb{R}}\frac{|x|^2}{2}\Delta ndx-\int_{\mathbb{T}^2\times \mathbb{R}}\frac{|x|^2}{2}\nabla (n\nabla c)dx\\
&=\int_{\mathbb{T}^2\times \mathbb{R}} n_0dx-\int_{\mathbb{T}^2\times \mathbb{R}}x\cdot n\left(\sum_{k\in \mathbb{Z}^3}\int_{\mathbb{T}^2\times \mathbb{R}}(x-y+k)\widetilde{E}(x-y+k)n(t,y)dy\right)dx.
\end{aligned}
\end{equation}
Combining the symmetry of $k$ and \eqref{eq:A.27}, one gets
$$
\begin{aligned}
&\int_{\mathbb{T}^2\times \mathbb{R}}x\cdot n\left(\sum_{k\in \mathbb{Z}^3}\int_{\mathbb{T}^2\times \mathbb{R}}(x-y+k)\widetilde{E}(x-y+k)n(t,y)dy\right)dx\\
=&\int_{\mathbb{T}^2\times \mathbb{R}}x\cdot n\left(\sum_{k\in \mathbb{Z}^3}\int_{\mathbb{T}^2\times \mathbb{R}}(x-y)\widetilde{E}(x-y+k)n(t,y)dy\right)dx\\
=&\sum_{k\in \mathbb{Z}^3}\int_{\mathbb{T}^2\times \mathbb{R}}\int_{\mathbb{T}^2\times \mathbb{R}}x\cdot (x-y)n(t,x)n(t,y)\widetilde{E}(x-y+k)dxdy\\
=&\frac{1}{2}\sum_{k\in \mathbb{Z}^3}\int_{\mathbb{T}^2\times \mathbb{R}}\int_{\mathbb{T}^2\times \mathbb{R}} (x-y)^2n(t,x)n(t,y)\widetilde{E}(x-y+k)dxdy\\
\geq& \frac{1}{2}\int_{\mathbb{T}^2\times \mathbb{R}}\int_{\mathbb{T}^2\times \mathbb{R}} (x-y)^2n(t,x)n(t,y)\widetilde{E}(x-y)dxdy\\
=&\frac{C_{d}}{2}\int_{\mathbb{T}^2\times \mathbb{R}}\int_{\mathbb{T}^2\times \mathbb{R}} \frac{1}{|x-y|}n(t,x)n(t,y)\left(\int_{0}^{\infty}e^{-s-\frac{|x-y|^2}{4s}}s^{\frac{1}{2}}ds\right)dxdy.
\end{aligned}
$$
For any $x,y\in \mathbb{T}^2\times \mathbb{R}$, and $|x-y|\leq R<1$, one has
$$
\int_{0}^{\infty}e^{-s-\frac{|x-y|^2}{4s}}s^{\frac{1}{2}}ds\geq \int_{0}^{\infty}e^{-s-\frac{R^2}{4s}}s^{\frac{1}{2}}ds\geq \int_{0}^{\infty}e^{-s-\frac{1}{4s}}s^{\frac{1}{2}}ds=C_{R}.
$$
Then we have
$$
\begin{aligned}
&\frac{C_{d}}{2}\int_{\mathbb{T}^2\times \mathbb{R}}\int_{\mathbb{T}^2\times \mathbb{R}} \frac{1}{|x-y|}n(t,x)n(t,y)\left(\int_{0}^{\infty}e^{-s-\frac{|x-y|^2}{4s}}s^{\frac{1}{2}}ds\right)dxdy\\
\geq& \frac{C_{d}C_{R}}{2R}\int_{|x-y|\leq R}n(t,x)n(t,y)dxdy\\
=&\frac{C_{d}C_{R}}{2R}\int_{\mathbb{T}^2\times \mathbb{R}}\int_{\mathbb{T}^2\times \mathbb{R}}n(t,x)n(t,y)dxdy-\frac{C_{d}C_{R}}{2R}\int_{|x-y|\geq R}n(t,x)n(t,y)dxdy,
\end{aligned}
$$
and
$$
\begin{aligned}
&\frac{C_{d}C_{R}}{2R}\int_{|x-y|\geq R}n(t,x)n(t,y)dxdy\\
\leq&\frac{C_{d}C_{R}}{2R}\int_{|x-y|\geq R}\frac{|x-y|^2}{R^2}n(t,x)n(t,y)dxdy\\
\leq&\frac{C_{d}C_{R}}{2R^3}\int_{|x-y|\geq R}|x-y|^2n(t,x)n(t,y)dxdy\\
\leq&\frac{C_{d}C_{R}}{R^3}\int_{\mathbb{T}^2\times \mathbb{R}}\int_{\mathbb{T}^2\times \mathbb{R}}(|x|^2+|y|^2)n(t,x)n(t,y)dxdy\\
\leq&\frac{2C_{d}C_{R}}{R^3}\int_{\mathbb{T}^2\times \mathbb{R}}\int_{\mathbb{T}^2\times \mathbb{R}}|x|^2n(t,x)n(t,y)dxdy.
\end{aligned}
$$
Therefore, we have
\begin{equation}\label{eq:A.29}
\begin{aligned}
&-\int_{\mathbb{T}^2\times \mathbb{R}}x\cdot n\left(\sum_{k\in \mathbb{Z}^3}\int_{\mathbb{T}^2\times \mathbb{R}}(x-y)\widetilde{E}(x-y+k)n(t,y)dy\right)dx\\
&\leq -\frac{C_{d}}{2}\int_{\mathbb{T}^2\times \mathbb{R}}\int_{\mathbb{T}^2\times \mathbb{R}} \frac{1}{|x-y|}n(t,x)n(t,y)\left(\int_{0}^{\infty}e^{-s-\frac{|x-y|^2}{4s}}s^{\frac{1}{2}}ds\right)dxdy\\
&\leq -\frac{C_{d}C_{R}}{2R}\int_{\mathbb{T}^2\times \mathbb{R}}\int_{\mathbb{T}^2\times \mathbb{R}}n(t,x)n(t,y)dxdy+\frac{C_{d}C_{R}}{2R}\int_{|x-y|\geq R}n(t,x)n(t,y)dxdy\\
&\leq -\frac{C_{d}C_{R}}{2R}\int_{\mathbb{T}^2\times \mathbb{R}}\int_{\mathbb{T}^2\times \mathbb{R}}n(t,x)n(t,y)dxdy+\frac{2C_{d}C_{R}}{R^3}\int_{\mathbb{T}^2\times \mathbb{R}}\int_{\mathbb{T}^2\times \mathbb{R}}|x|^2n(t,x)n(t,y)dxdy\!\\
&=-\frac{C_{d}C_{R}}{2R}\left(\int_{\mathbb{T}^2\times \mathbb{R}}n(t,x)dx\right)^2+\frac{2C_{d}C_{R}}{R^3}\left(\int_{\mathbb{T}^2\times \mathbb{R}}n(t,y)dy\right)\int_{\mathbb{T}^2\times \mathbb{R}}|x|^2n(t,x)dxdy.
\end{aligned}
\end{equation}
Combining \eqref{eq:A.28} and \eqref{eq:A.29}, we obtain
$$
\frac{d}{dt}I(t)\leq M-\frac{C_{d}C_{R}}{2R}M^2+\frac{4C_{d}C_{R}}{R^3}MI(t).
$$
For $\epsilon_0> 0$ is small enough, choose
$$
R=\epsilon_0C_{d}C_{R}M<1,
$$
then we have
$$
\frac{d}{dt}I(t)\lesssim M\left( \frac{C}{\epsilon_0C^2_dC^2_{R}M^3}I(t)-1\right).
$$
If $I(0)$ is small enough, then  for any $t\geq 0$, one has
$$
\frac{d}{dt}I(t)\lesssim M\left( \frac{C}{\epsilon_0C^2_dC^2_{R}M^3}I(0)-1\right)<0,
$$
this leads to a contradiction with $I(t)$ cannot be negative. %This completes the proof of Lemma \ref{lem:A.5}.
We finish the proof of Lemma \ref{lem:A.5}.
\end{proof}

%\vskip .05in
%\section*{Acknowledgement}
%The authors would like to thank Dr.Yuanyuan Feng for helpful discussions. The work of the first author was partially %supported by  Shanghai Science and Technology Innovation Action Plan (Grant No.21JC1403600). The work of the second %author was partially supported by the National Natural Science Foundation of China (Grant %No.12271357,11831011,12161141004) and Shanghai Science and Technology Innovation Action Plan (Grant No.21JC1403600).

\noindent \textbf{Acknowledgement.} The authors would like to thank Dr.Yuanyuan Feng for helpful discussions. The work of the first author was partially supported by  Shanghai Science and Technology Innovation Action Plan (Grant No.21JC1403600). The work of the second author was partially supported by the National Natural Science Foundation of China (Grant No.12271357,11831011,12161141004) and Shanghai Science and Technology Innovation Action Plan (Grant No.21JC1403600).

\bibliographystyle{abbrv}
\bibliography{3DKS-flow-ref}

\end{document}